\documentclass[11pt]{amsart}
\usepackage{amsmath, amssymb, mathrsfs, mathtools}

\usepackage{stmaryrd}
  \usepackage{graphics} %% add this and next lines if pictures should be in esp format
  \usepackage{graphicx}
  \usepackage{epsfig} %For pictures: screened artwork should be set up with an 85 or 100 line screen
\usepackage[margin=1.1in]{geometry}

\newtheorem{theorem}{Theorem}[section]

\newtheorem{lemma}[theorem]{Lemma}
\newtheorem{claim}[theorem]{Claim}
\newtheorem{proposition}[theorem]{Proposition}
\theoremstyle{definition}

\newtheorem{remark}[theorem]{Remark}

\newcommand{\eps}{\varepsilon}

\newcommand{\Rb}{\mathbb{R}}
\newcommand{\RR}{\mathbb{R}}
\newcommand{\ZZ}{\mathbb{Z}}

\newcommand{\be}{\begin{equation}}
\newcommand{\ee}{\end{equation}}
\newcommand{\pa}{\partial}
\newcommand{\Ce}{\mathcal C^e}

\newcommand{\tu}{\tilde{u}}
\newcommand{\tf}{\tilde{f}}

\newcommand{\tGamma}{\widetilde{\Gamma}}

\newcommand{\lambdabf}{\boldsymbol{\lambda}}

\newcommand{\HHH}{\mathcal{H}}

\newcommand{\indic}{1\!\!1}

\numberwithin{equation}{section}

\title[Channels of energy for critical wave in the degenerate case] %use the shortened version of the full title
{On channels of energy for the radial linearised energy critical wave equation in the degenerate case}

\author[C.~Collot]{Charles Collot}
\address{CNRS and AGM (UMR 8088) laboratory of CY Cergy Paris Universit\'e, 2 rue Adolphe Chauvin, 95300 Pontoise, France}
\email{ccollot@cyu.fr}
\author[T.~Duyckaerts]{Thomas Duyckaerts}
\address{LAGA (UMR 7539), Universit\'e Sorbonne Paris Nord, Institut Galil\'ee, 99 avenue Jean-Baptiste Cl\'ement, 93430 Villetaneuse, France}
\email{duyckaer@math.univ-paris13.fr}
\author[C.~Kenig]{Carlos Kenig}
\address{University of Chicago, Department of Mathematics, 5734 University Avenue, Chicago, IL 60637-1514, USA}
\email{ckenig@uchicago.edu}
\author[F.~Merle]{Frank Merle}
\address{Institut des Hautes \'Etudes Scientifiques, and AGM (UMR 8088) laboratory of CY Cergy Paris Universit\'e 2 rue Adolphe Chauvin, 95300 Pontoise, France}
\email{frank.merle@cyu.fr}

\thanks{ 
\today}
\begin{document}

\begin{abstract}
Channels of energy estimates control the energy of an initial data from that which it radiates outside a light cone. For the linearised energy critical wave equation they have been obtained in the radial case in odd dimensions, first in $3$ dimensions in \cite{DuKeMe13}, then for the general case in \cite{DuKeMe20}. We consider even dimensions, for which such estimates are known to fail \cite{CoKeSc14}. We propose a weaker version of these estimates, around a single ground state as well as around a multisoliton. This allows us in \cite{CoDuKeMe22Pv1} to prove the soliton resolution conjecture in six dimensions.
\end{abstract}

\maketitle

\section{Introduction and results}

We consider the linearised energy critical wave equation around the ground state in $N\geq 3$ dimensions:
\begin{equation}
 \label{LW}
  \left\{
 \begin{aligned}
& \partial_t^2u_L-\Delta u_L+V u_L=0\\
& \vec{u}_{L\restriction t=0}=(u_0,u_1),
\end{aligned} \right.
\end{equation}
where $(u_0,u_1)\in \mathcal H=\dot H^1 \times L^2(\mathbb R^N)$ and
$$
V=-\frac{N+2}{N-2}W^{\frac{4}{N-2}}, \qquad W(x)= \left(1+\frac{|x|^2}{N(N-2)}\right)^{-\frac{N-2}{2}}.
$$
This is the linearised equation for 
\begin{equation}
 \label{NLWabs}
 \partial_t^2u-\Delta u=|u|^{\frac{4}{N-2}}u 
\end{equation} 
around the stationary solutions $\pm W$. By standard arguments, recalled in Section \ref{subsec:channelsunsoliton}, Equation \eqref{LW} is globally well-posed in $\HHH$, and moreover the outer radiated energy
\begin{equation} \label{id:exteriorenergy}
E_{\textup{out}}=E_{\textup{out}}^-+E_{\textup{out}}^+, \qquad E_{\textup{out}}^\pm=\lim_{t\to \pm \infty}  \int_{|x|> |t|} |\nabla_{t,x}u_L(t,x)|^2dx,
\end{equation}
where $\nabla_{t,x}f =(\pa_t f,\nabla f)$, is well-defined for both time directions $t\to \pm \infty$. For the free wave equation
\begin{equation}
 \label{FW}
   \left\{
 \begin{aligned}
& \partial_t^2u_F-\Delta u_F=0,\\
& \vec{u}_{F\restriction t=0}=(u_0,u_1),
\end{aligned} \right.
\end{equation} 
it controls the total energy of the initial data in odd dimensions \cite{DuKeMe11a,DuKeMe12}:
\begin{equation} \label{bd:channelsodddimensions}
\| u_0\|_{\dot H^1}+\| u_1\|_{L^2}\lesssim \sqrt{E_{\textup{out}}} \qquad \mbox{if }N\geq 3 \mbox{ is odd}
\end{equation} 
but only for half the data in even dimensions \cite{CoKeSc14}:
\begin{equation} \label{free_exterior}
\| u_0\|_{\dot H^1}\lesssim \sqrt{E_{\textup{out}}} \quad \mbox{if }N\equiv 4 \mod 4 \qquad \mbox{and} \qquad \| u_1\|_{L^2}\lesssim \sqrt{E_{\textup{out}}} \quad \mbox{if }N\equiv 6 \mod 4
\end{equation} 
while the full estimate \eqref{bd:channelsodddimensions} is known to fail \cite{CoKeSc14}. Extensions to domains of the form $\{|x|>R+|t|\}$ for $R>0$, obtained in \cite{KeLaLiSc15,DuKeMaMe21P,LiShenWei21P}, will be used in the present paper. Other recent results on the asymptotic behaviour of linear waves can be found in \cite{Delort2021P,cote2021concentration,li2022asymptotic}.

For the linearised wave equation \eqref{LW}, two natural counter examples to estimates like \eqref{bd:channelsodddimensions} and \eqref{free_exterior} are $\Lambda W$ and $t\Lambda W$ (for $N\geq 5$), as they are non-radiative i.e. $E_{\textup{out}}=0$. Here $\Lambda W=x.\nabla W+\frac{N-2}{2}W$ is in the radial kernel of $-\Delta +V$. For odd dimensions, the strong estimates \eqref{bd:channelsodddimensions} and their extensions to domains $\{|x|>R+|t|\}$ for $R>0$ allowed the authors of \cite{DuKeMe20} to extend this estimate to Equation \eqref{LW} in the radial case:
\begin{equation} \label{bd:channelsodddimensionslinearised}
\| \Pi_{\dot H^1}^\perp u_0\|_{\dot H^1}+\| \Pi_{L^2}^\perp u_1\|_{L^2}\lesssim \sqrt{E_{\textup{out}}} \qquad \mbox{if }N\geq 3 \mbox{ is odd}.
\end{equation} 
Above, we used the projectors
\begin{equation} \label{def:PiH1PiL2}
\Pi_{\dot H^1}^\perp=\Pi_{\dot H^1} (\text{Span}(\Lambda W))^\perp, \qquad \Pi_{L^2}^\perp=\Pi_{L^2} (\text{Span}(\Lambda W))^\perp,
\end{equation}
where, for $H$ a Hilbert space, and $E$ a closed linear subspace of $H$, we denote by $\Pi_H(E)$ the orthogonal projection onto $E$ in $H$.

In this paper, we consider the even dimensional case. We focus on six and eight dimensions, since we believe their extensions to other even dimensions only to be a technical refinement of the present arguments. Our first result is a weaker version of \eqref{bd:channelsodddimensionslinearised} involving the space $Z_\alpha(\mathbb R^N)$ associated to the following norm for $\alpha \in \mathbb R$:
\begin{equation}
\label{defZ}
\| f\|_{Z_{\alpha}}=\sup_{R>0} \ \frac{R^{-\frac N2-\alpha}}{\langle \log R \rangle}  \left(\int_{R<|x|<2R} f^2dx \right)^{\frac 12}.
\end{equation} 
Note that, $L^2(\mathbb R^6)\subset Z_{-3}(\mathbb R^6)$ with $ \| f \|_{Z_{-3}(\mathbb R^6)} \lesssim \| f \|_{L^2(\mathbb R^6)}$, so that $Z_{-3}(\mathbb R^6)$ is a logarithmic weakening at $0$ and $\infty$ of $L^2(\mathbb R^6)$. 

\begin{theorem}[Channels of energy around the ground state] \label{th:channelsunsoliton}
Assume $N=6$. There exists $C>0$ such that any radial solution $u_L$ of \eqref{LW} satisfies:
\be \label{bd:channelunsoliton}
\big\| \Pi_{L^2}^\perp u_1\big\|_{L^2}+\big\|\nabla \Pi_{\dot H^1}^\perp u_0 \big\|_{Z_{-3}} \leq C \sqrt{E_{\textup{out}}}.
\ee
% $$
% \| \Pi_{L^2}^\perp \pa_t u_1\|_{L^2}^2+\| \nabla \Pi_{\dot H^1}^\perp u_0 \|_{Z_{-3}}^2\leq C \sum_{\pm } \lim_{t\rightarrow \pm \infty}  \int_{r\geq |t|} |\nabla_{t,x}u(t,x)|^2dx.
% $$
\end{theorem}

The weaker estimate \eqref{bd:channelunsoliton} is sufficient to allow the authors to show in \cite{CoDuKeMe22Pv1} the soliton resolution and the inelasticity of collisions of solitons for Equation \eqref{NLWabs} in the radial case. We believe the extension of \eqref{bd:channelunsoliton} to even dimensions $N\equiv 6\mod 4$ to be a technical refinement of the present proof. An analogue in $N=8$ dimensions is given in Theorem \ref{th:channelsunsoliton8d} in Section \ref{sec:8d}, which we similarly believe to extend to even dimensions $N\equiv 4\mod 4$.

\begin{remark}

An analog of the bound on $u_0$ in \eqref{bd:channelunsoliton} for solutions of the free wave equation \eqref{FW} is also valid. Indeed, let $\Phi$ be a smooth radial function with compact support included in $\{r>0\}$ such that $\int \frac{1}{r^4}\Phi dx \neq 0$. Then there exists a constant $C>0$ such that for all solution $u_F$ of \eqref{FW} with initial data $u_0$
\begin{equation}
\label{good_estimate}
\int \nabla u_0\cdot\nabla \Phi \,dx=0 \Longrightarrow
\big\| u_0 \big\|_{Z_{-2}} \lesssim
\big\| \nabla u_0 \big\|_{Z_{-3}} \leq C \sqrt{E_{\textup{out}}},
\end{equation} 
with a proof that is similar to the proof of \eqref{bd:channelunsoliton}. Note that this does not imply the estimate on the projection: $
\big\| \Pi_{\dot{H}^1}\big(\{\Phi\}^{\bot}\big) u_0 \big\|_{Z_{-2}} \lesssim \sqrt{E_{\textup{out}}}$,
which is false.
\end{remark}

\begin{remark}
\label{re:channelsunsoliton}
The stronger estimates \eqref{bd:channelsodddimensionslinearised} fail in six dimensions. Indeed, defining for $\alpha \in \mathbb R$:
$$
\| f\|_{\widetilde{Z}_{\alpha}(\mathbb R^6)}=\sup_{R>0} R^{-3-\alpha} \left(\int_{R<|x|<2R} f^2dx \right)^{\frac 12},
$$
we show in Appendix \ref{ap:counterexample} that the estimate
 \begin{equation}
\label{bad_estimate''}
\| \nabla u_0  \|_{\tilde Z_{-3}} \lesssim \sqrt{E_{\textup{out}}}
\end{equation} 
fails for solutions of \eqref{FW}. This implies that both the projection and the logarithmic loss in \eqref{good_estimate} are necessary, i.e. that the estimates:
\begin{equation*}
\| \nabla \Pi_{\dot{H}^1}(V^{\bot}) u_0 \|_{\widetilde{Z}_{-3}} \lesssim  \sqrt{E_{\textup{out}}},\qquad  \| \nabla u_0 \|_{Z_{-3}} \lesssim \sqrt{ E_{\textup{out}}}, 
\end{equation*} 
where $V$ is any finite dimensional subspace of $\dot{H}^1$ are both false, as they would imply \eqref{bad_estimate''} by rescaling the solution and letting the scaling parameter go to $0$ or $\infty$. For the same reason, the logarithmic loss is necessary in \eqref{bd:channelunsoliton}, i.e. it is not possible to replace $\|\nabla \Pi_{\dot{H}^1}^{\bot}u_0\|_{Z_{-3}}$ by 
$\|\nabla \Pi_{\dot{H}^1}^{\bot}u_0\|_{\widetilde{Z}_{-3}}$ (or even the smaller quantity\footnote{Using Cauchy-Schwarz and the formula $f(r)=-\int_r^{\infty} \partial_rf(s)ds$, one can prove the following variant of Hardy's inequality: $ \|u\|_{Z_{-2}}\lesssim \|\nabla u\|_{Z_{-3}}$ and $ \|u\|_{\tilde{Z}_{-2}}\lesssim \|\nabla u\|_{\tilde{Z}_{-3}}$,  for any $u\in \dot{H}^1$ radial.} $\|\Pi_{\dot{H}^1}^{\bot}u_0\|_{\widetilde{Z}_{-2}}$) in this inequality.
\end{remark}

Our second result extends Theorem \ref{th:channelsunsoliton} to the linearised equation around a multisoliton:
\begin{equation}
 \label{id:linearmultisoliton} 
 \left\{
 \begin{aligned}
 \partial_t^2u-\Delta u+V_{\lambdabf}u=0,\\
 \vec{u}_{\restriction t=0}=(u_0,u_1)\in \mathcal H.
 \end{aligned}\right.
\end{equation}
Above, $\lambdabf \in\Lambda_J=\left\{\lambdabf=(\lambda_1,...,\lambda_J)\in (0,\infty)^J, \ \lambda_J< \lambda_{J-1}<...<\lambda_{1} \right\}$ for some $J\in \mathbb N$ and
$$
V_{\lambdabf}=  \sum_{j=1}^J V_{(\lambda_j)}
$$
where we use the following notations for $\dot H^1$ and $L^2$ rescalings
$$
f_{(\lambda)}(x)=\frac{1}{\lambda^2}f\left(\frac{x}{\lambda}\right) \qquad \mbox{and} \qquad f_{[\lambda]}(x)=\frac{1}{\lambda^3}f\left(\frac{x}{\lambda}\right).
$$
By standard arguments, recalled in Section \ref{subsec:channels_multi}, Equation \eqref{id:linearmultisoliton} is globally well-posed and the outer radiated energy \eqref{id:exteriorenergy} is well-defined.

We define for all $\lambdabf \in \Lambda_J$ the scale separation parameter:
$$
\gamma (\lambdabf)=\max_{1\leq j\leq J-1}\frac{\lambda_{j+1}}{\lambda_j}.
$$
and for $\alpha \in \mathbb R$ the Banach space $Z_{\alpha,\lambdabf}$ associated to the norm
\begin{gather} \label{def:Zalpha}
\| f\|_{Z_{\alpha,\lambdabf}}=\sup_{R>0} \ \frac{R^{-3-\alpha}}{\inf_{1\leq j \leq J} \ \langle \log \frac{R}{\lambda_j} \rangle}  \left(\int_{R<|x|<2R} f^2dx \right)^{\frac 12}.
\end{gather}
Note that for any $J\in \mathbb N$ and $\lambdabf\in \Lambda_J$ there holds $L^2(\mathbb R^6)\subset Z_{-3,\lambdabf}(\mathbb R^6)$ with $ \| f \|_{Z_{-3,\lambdabf}(\mathbb R^6)} \lesssim \| f \|_{L^2(\mathbb R^6)}$, where the implicit constant is independent of $J$ and $\lambdabf$. Thus, $Z_{-3,\lambdabf}$ is a weakening of $L^2$, with a loss that is logarithmic in the distance to the closest soliton. We define:
\begin{align*}
&\Pi_{\dot H^1,\lambdabf}^\perp= \Pi_{\dot H^1}  \left( \text{Span}((\Lambda W)_{(\lambda_j)})_{1\leq j \leq J} \right)^\perp, \qquad \Pi_{L^2,\lambdabf}^{\bot}= \Pi_{L^2}  \left( \text{Span}((\Lambda W)_{[\lambda_j]})_{1\leq j \leq J} \right)^\perp.
\end{align*}

\begin{theorem}[Channels of energy around a multisoliton] \label{pr:channels}

Assume $N=6$. For any $J\in \mathbb N$, there exist $\gamma^*,C>0$ such that for any $\lambdabf \in \Lambda_J$ with $\gamma(\lambdabf)\leq \gamma^*$ if $u$ is a radial solution of \eqref{id:linearmultisoliton} then:
\be \label{bd:channelsmultisoliton}
\| \Pi_{L^2,\lambdabf}^\perp \, u_1 \|_{L^2}+\| \nabla \Pi_{\dot H^1,\lambdabf}^\perp \, u_0 \|_{Z_{-3,\lambdabf}}\leq C\left(\sqrt{E_{\textup{out}}} \ +\gamma (\lambdabf) \| (u_0,u_1)\|_{\mathcal H}\right).
\ee
\end{theorem}

\begin{remark}

While in odd dimensions, an analogue of \eqref{bd:channelsmultisoliton} with the second term in the left-hand side of \eqref{bd:channelsmultisoliton} replaced by the $\dot H^1$ norm is true, see Corollary 3.3 in \cite{DuKeMe19Pb}, this cannot hold true in six dimensions, even with the weaker norm $\|\Pi_{\dot H^1,\lambdabf}^\perp u_0\|_{\widetilde{Z}_{-2}}$ in the left-hand side of \eqref{bd:channelsmultisoliton}, 
see Remark \ref{re:channelsunsoliton}.

\end{remark}

The paper is organised as follows. Section \ref{sec:basicestimates} contains some notation and basic estimates. Then Section \ref{subsec:channelsunsoliton} is devoted to the proof of Theorem \ref{th:channelsunsoliton} ; a generalised estimate for odd in time solutions on domains $\{|x|>R+|t-\overline t|\}$ for $R>0$ is first proved in Lemma \ref{L:channelsoddunsoliton}, and then is used to prove the analogue but weakened generalised estimate for even solutions in Lemmas \ref{L:estimateufromutunsoliton} and \ref{lem:estimateufromutunsoliton2}. Based on these generalised estimates, the proof of Theorem \ref{pr:channels} is then given in Section \ref{subsec:channels_multi}. Section \ref{sec:8d} treats the eight dimensions, see Theorem \ref{th:channelsunsoliton8d}. A counter-example to the non-weakened estimates is given in Appendix \ref{ap:counterexample}, and a few technical estimates are given in Appendix \ref{A:estimates}.

\section{Acknowledgements}

This work was supported by the National Science Foundation [DMS-2153794 to C.K.]; the CY Initiative of Excellence Grant
"Investissements d'Avenir" [ANR-16-IDEX-0008 to C.C. and F.M.]; and the France and Chicago Collaborating in the Sciences [FACCTS award \#2-91336 to C.C. and F.M.]

\section{Notations and estimates for free waves} \label{sec:basicestimates}
If $u$ is a function of space and time, we write $\vec{u}=(u,\partial_tu)$. 

For $R\geq 0$, $p\in (1,\infty)$ we write 
\begin{equation*}
\|u\|^p_{L^p_R}=\int_{R}^{\infty} (u(r))^pr^5dr,\quad 
\|u\|^2_{\dot{H}^1_R}=\int_{R}^{\infty}(\partial_ru(r))^2r^5dr,
\end{equation*}
We let
$$
\mathcal H_R=\dot H^1_R\times L^2_R.
$$

\begin{remark}
\label{R:ext}
Let $R>0$ and $u$ be a radial function defined for $r>R$. Then the extension $u_R$ of $u$ defined by
$$ u_R(r)=u(r),\; r>R,\quad u_R(r)=3u(2R-r)-2u(3R-2r), \; 0<r<R,$$
satisfies, for all $p\geq 1$
$$\|u_R\|_{L^p(\RR^6)}\leq C\|u\|_{L^p_R},\quad \|\partial_ru_R\|_{L^p(\RR^6)}\leq C\|\partial_r u\|_{L^p_R}$$
where the constant $C$ is independent of $u$, $p$ and $R$.
\end{remark}
For $(t,R)\in \Rb\times (0,\infty)$, we let 
$$\Ce_{t,R}=\Big\{ \left(\overline{t},\overline{r}\right)\in \RR\times (0,\infty)\,:\, \overline{r}>R+|t-\overline{t}|\Big\}$$
be the exterior cone. The exterior energy is 
$$\|u\|_{E_{t,R}}=\sup_{\overline{t}\in \RR} \left\|\vec{u}\left( \overline{t} \right)\right\|_{\dot{H}^1_{R+|t-\overline{t}|}\times L^2_{R+|t-\overline{t}|}}.$$

We introduce the Strichartz norms:
\begin{equation}
 \|u\|_{L^p_tL^q_r(r>R+|t-\overline t|)}=\left( \int_{\overline{t}\in \RR} \left( \int_{r>R+|t-\overline{t}|}\left|u(\overline{t},r)\right|^qr^5dr \right)^{\frac pq}d\overline{t} \right)^{\frac 1p}=\|u\|_{L^pL^q\left( \Ce_{t,R} \right)}.
\end{equation} 
We recall the Strichartz estimates: if $(u_0,u_1)\in \dot{H}^1\times L^2(\RR^6)$ and 
$$u(t)=\cos t\sqrt{-\Delta}u_0+\frac{\sin t\sqrt{-\Delta}}{\sqrt{-\Delta}}u_1+\int_0^t \frac{\sin(t-t')}{\sqrt{-\Delta}} f(t')dt',$$
with $f\in L^1_tL^2_r(\RR^6)$, we have
\begin{equation}
\label{CK1}
 \sup_{t\in \RR}\left\|\vec{u}(t)\right\|_{\dot{H}^1\times L^2}+\|u\|_{L^2_tL^{4}_r}
 \lesssim \|(u_0,u_1)\|_{\dot{H}^1\times L^2}+\|f\|_{L^1_tL^2_r}.
\end{equation}
We next state the analog of \eqref{CK1} for exterior cones
\begin{lemma}
 \label{L:CK1.3}
 Assume that $(u_0,u_1)\in \dot{H}^1\times L^2(\RR^6)$, $u$ is the solution of $\partial_t^2u-\Delta u=f$, $u_{\restriction t=t_0}=u_0$, $\partial_tu_{\restriction t=t_0}=u_1$ with $f\in L^1_tL^2_r$ (all radial). Then, for any $R_0>0$ (with the implicit constant independent of $(t_0,R_0)$, we have
$$
 \left\|\vec{u}(t)\right\|_{E_{t_0,R_0}}+\|u\|_{L^2_tL^{4}_r\left( \Ce_{t_0,R_0} \right)}\lesssim \|(u_0,u_1)\|_{\dot{H}^1\times L^2}+\|f\|_{L^1_tL^2_r\left( \Ce_{t_0,R_0} \right)}.
$$
\end{lemma}
\begin{proof}
 Without loss of generality, we assume $t_0=0$. By \eqref{CK1} above, we have 
$$
 \sup_{t\in \RR}\left\|\vec{u}(t)\right\|_{\dot{H}^1\times L^2}+\|u\|_{L^2_tL^{4}_r}
 \lesssim \|(u_0,u_1)\|_{\dot{H}^1\times L^2}+\|f\|_{L^1_tL^2_r}.$$
 Recall that by finite speed of propagation, the solution $u$ restricted to $\Ce_{0,R_0}$ depends only on $(u_0,u_1)$ on $r>R_0$, and $f$ on $\Ce_{0,R_0}$. Clearly $\tf=f\indic_{\Ce_{0,R_0}}\in L^1_tL^2_r$. Let now $(\tu_0,\tu_1)$ be any extension of $(u_0,u_1)$ restricted to $r>R_0$, $(\tu_0,\tu_1)\in \dot{H}^1\times L^2$, with $\left\|(\tu_0,\tu_1)\right\|_{\dot{H}^1\times L^2} \lesssim \|(u_0,u_1)\|_{\dot{H}^1_{R_0}\times L^2_{R_0}}$. Let $\tilde{u}$ solve $\Box \tu=\tf$. By finite speed of propagation, $u=\tu$ on $\Ce_{0,R_0}$. But then, by definition of the restriction norms, our left-hand side in the statement is smaller than the corresponding global norms for $\tu$, which can be bounded by \eqref{CK1}. Hence, the left-hand bounded by $C\left\{ \|(u_0,u_1)\|_{\dot{H}^1_{R_0}\times L^2_{R_0}}+\|\tf \|_{L^1_tL^2_r}\right\}$ which is the desired result.
\end{proof}
\begin{proposition}[Channels in 6d, with right-hand side]
 \label{P:CK1.4}
 Let $u\in \RR\times \RR^6$, solve 
\begin{equation*}
 \Box u=f,\quad u_{\restriction t=0}=u_0,\quad \partial_tu_{\restriction t=0}=u_1,
\end{equation*} 
$f,u_0,u_1$ radial. Fix $R>0$ and write $u_1=\frac{c_1}{r^4}+u_1^{\bot}$, where $\int_{R}^{\infty} u_1^{\bot}\frac{1}{r^4}r^5dr=0$. Then,
\begin{equation}
 \label{CK1.1}
 \left\|u_1^{\bot}\right\|_{L^2_R}\lesssim \left[\lim_{t\to\infty} \left\|\vec{u}(t)\right\|_{\dot{H}^1_{R+|t|}\times L^2_{R+|t|}}+\|f\|_{L^1_tL^2_r\left( \Ce_{0,R} \right)}\right].
\end{equation}
 \end{proposition}
\begin{proof}
 Let $\tu$ be the solution of the homogeneous wave equation with initial data $(u_0,u_1)$, and let $v$ be the solution of the inhomogeneous equation with right-hand side $f\indic_{\Ce_{0,R}}$, and $(0,0)$ initial data. Then $u=\tu+v$ on $\Ce_{0,R}$, by finite speed of propagation. By \cite[Proposition 3.8]{DuKeMaMe21P}, we have  
 \begin{multline*}
\left\|u_1^{\bot}\right\|_{L^2_R}=\left\|\tu_1^{\bot}\right\|_{L^2_R}\leq \frac{20}{3}\lim_{t\to\infty} \left\| \left(\vec{u}-\vec v\right)(t)\right\|_{\dot{H}^1_{R+|t|}\times L^2_{R+|t|}}\\
\leq \frac{20}{3} \left( \lim_{t\to\infty} \left\|\vec{u}(t)\right\|_{\dot{H}^1_{R+|t|}\times L^2_{R+|t|}}+\sup_{t}\left\|\vec{v}(t)\right\|_{\dot{H}^1_{R+|t|}\times L^2_{R+|t|}}\right)\\
\leq \frac{20}{3} \lim_{t\to\infty} \left\|\vec{u}(t)\right\|_{\dot{H}^1_{R+|t|}\times L^2_{R+|t|}}+C\|\tf\|_{L^1_tL^2_r} \\
\leq \frac{20}{3} \lim_{t\to\infty} \left\|\vec{u}(t)\right\|_{\dot{H}^1_{R+|t|}\times L^2_{R+|t|}}+C\|f\|_{L^1_tL^2_r\left( \Ce_{0,R} \right)}, 
 \end{multline*}
 by  \eqref{CK1} and Lemma \ref{L:CK1.3}.
\end{proof}

\section{Channels of energy around the ground state} \label{subsec:channelsunsoliton}

It is easy to check that the solution $u$ of \eqref{LW} is globally well-posed in $\HHH$. Indeed, the local well-posedness can be proved by Strichartz estimates \eqref{CK1} and the fact that $W$ is in $L^4$. The global well-posedness follows from the linearity of the equation. Furthermore, using finite speed of propagation, the Strichartz estimates of Lemma \ref{L:CK1.3} and that $\lim_{R\to \infty} \| \indic (R<|t|<|x|-R)W \|_{L^2L^4}=0$, the outer radiated energy \eqref{id:exteriorenergy} is well-defined.

\subsection{Channels of energy around the ground state for odd in time solutions}

We first establish, for odd in time solutions, a more general estimate result than that of Theorem \ref{th:channelsunsoliton} for $\pa_t u$, valid for the exterior of any wave cone $\{r\geq R+|t|\}$. The zeros of $-\Delta +V$ play an important role. We recall that there holds
\begin{equation}\label{id:LambdaWGammazeroes}
-\Delta \Lambda W+V\Lambda W=0 \qquad \mbox{and} \qquad -\Delta \Gamma+V\Gamma =0,
\end{equation}
where the second function $\Gamma(r)=-\Lambda W(r)\int_{1}^r s^{-5}(\Lambda W(s))^{-2}ds$ satisfies the Wronskian relation:
\begin{equation} \label{id:LambdaWGammawronskian}
\forall r>0, \qquad \Gamma (r)\frac{d \Lambda W(r)}{dr}-\Lambda W(r) \frac{d \Gamma(r)}{dr}=r^{-5}
\end{equation}
and the asymptotics for some $c,c'\neq 0$:
\begin{equation}\label{id:asymptoticGamma}
\Gamma(r)= cr^{-4}+O(r^{-2})\qquad \mbox{as }r\rightarrow 0, \qquad \mbox{and }\Gamma(r)=c'+O(r^{-2}) \quad \mbox{as }r\rightarrow \infty,
\end{equation}
which propagates for higher order derivatives:
\begin{equation}\label{id:asymptoticparGamma}
\frac{d\Gamma(r)}{dr}= -4cr^{-5}+O(r^{-3})\qquad \mbox{as }r\rightarrow 0, \qquad \mbox{and }\frac{d\Gamma(r)}{dr}=dr^{-3}+O(r^{-5})\quad \mbox{as }r\rightarrow \infty,
\end{equation}
where $d\neq 0$.
For $\chi_{0}$ a smooth cut-off function such that $\chi_{0}(r)=1$ for $r\leq 10$ and $\chi_{0}(r)=0$ for $r\geq 12$ we define $\tGamma$ as the truncation of the $\Gamma$ profile:
\begin{equation}\label{id:deftildeGamma}
\tGamma(r)=\chi_{0}(r)\Gamma (r)
\end{equation}
We then define for $R\geq 0$ the projection $\Pi^\perp_{L^2_R} u$ as follows:
\begin{equation}\label{id:defpiL2R1}
\Pi^\perp_{L^2_R}=\left\{ \begin{array}{l l} \Pi_{L^2_R} \left( \text{Span}(\Lambda W) \right)^\perp \qquad \mbox{for }R=0 \mbox{ or }R\geq 1, \\ \Pi_{L^2_R} \left( \text{Span}( \Lambda W,\tGamma) \right)^\perp \qquad \mbox{for }0<R< 1.  \end{array} \right.
\end{equation}
This section is devoted to the proof of the following result.

% OK

\begin{lemma}[Channels of energy for the odd in time part of solutions around a soliton] \label{L:channelsoddunsoliton}
There exists $C>0$ (independent of $R$) such that the following holds true for all $R\geq 0$. Let $u_L$ be a radial solution of \eqref{LW} for $r\geq R+|t|$, with $\vec u_L(0)\in \mathcal H_R$. Then:
$$
\| \Pi_{L^2_R}^\perp u_1 \|_{L^2}^2\leq C \sum_{\pm } \lim_{t\rightarrow \pm \infty}  \int_{r\geq R+|t|} |\nabla_{t,x}u_L|^2dx.
$$
\end{lemma}
The proof combines two quantitative lower-bounds, proved in \cite{DuKeMaMe21P,CoKeSc14}, namely Proposition \ref{P:CK1.4}, and that the bound \eqref{free_exterior} holds true for all solutions of \eqref{FW}, and the following rigidity result:
\begin{lemma}[Rigidity for odd in time non-radiative linear waves around a soliton] \label{L:rigiditylinearoddintime}

  Let $R\geq 0$, and $u_L$ be a solution of \eqref{LW}. Assume
  \begin{equation}\label{NonRad} \sum_{\pm}\lim_{t\to\pm\infty}\int_{|x|>R+|t|} |\nabla_{t,x}u_L(t,x)|^2dx=0.\end{equation}
  Then there exists $c \in \mathbb R$ such that $u_1(r)=c\Lambda W(r)$ for all $r\geq R$. 
  
\end{lemma}

\begin{proof}

Up to replacing $u_L$ by $\frac{1}{2}(u_L(t)-u_L(-t))$, we can assume that $u_0=0$. We fix a large $R'>0$ and choose $c\in \mathbb R$ such that 
$$ \int_{R'}^{\infty} (u_1-c\Lambda W)\frac{1}{r^4} r^5dr=0.$$
The solution of \eqref{LW} with initial data $(0,u_1-c\Lambda W)$ is $u_L(t)-ct\Lambda W$. By Proposition \ref{P:CK1.4},
$$ \left\|u_1-c\Lambda W\right\|_{L^2_{R'}}\lesssim \left\| V(u_L-ct\Lambda W)\right\|_{L^1L^2(r>R'+|t|)}.$$
Furthermore, by Strichartz estimates,
$$ \|u_L-ct\Lambda W\|_{L^2L^4(r>R'+|t|)}\lesssim \left\| V(u_L-ct\Lambda W)\right\|_{L^1L^2(r>R'+|t|)}+\|u_1-c\Lambda W\|_{L^2_{R'}}.$$
Combining, we obtain
\begin{multline}
\label{UL_LambdaW}
\|u_L-ct\Lambda W\|_{L^2L^4(r>R'+|t|)}\lesssim \left\| V(u_L-ct\Lambda W)\right\|_{L^1L^2(r>R'+|t|)}\\
\lesssim \| V\|_{L^2L^4(r>R'+|t|)}\left\|(u_L-ct\Lambda W)\right\|_{L^2L^4(r>R'+|t|)}.
\end{multline}
Since $V=-2W\in L^2_tL^4_x(r\geq |t|)$, we have 
$$ \lim_{R'\to\infty} \|V\|_{L^2L^4(r>R'+|t|)}=0,$$
and thus \eqref{UL_LambdaW} implies, chosing $R'$ large enough,
$$u_L(t,r)=ct\Lambda W(r),\quad r>R'+|t|.$$
As a consequence of \eqref{NonRad} and the propagation of the support for solutions of \eqref{LW} with compactly supported initial data (see \cite[Proposition 4.7]{DuKeMe20}), we obtain that $u_1(r)=c\Lambda W(r)$ for all $r\geq R$.
\end{proof}

We can now give the proof of Lemma \ref{L:channelsoddunsoliton}.
\begin{proof}[Proof of Lemma \ref{L:channelsoddunsoliton}]
We reason by contradiction, and assume that there exist sequences $(R_n)_n$ of radii and $(u_{L,n})_n$ of solutions of \eqref{LW},
with initial data $(0,u_{1,n})$ (up to replacing $u_{L,n}$ by $\frac 12 (u_{L,n}(t)-u_{L,n}(-t))$) such that $u_{1,n}(r)=0$ for $r\leq R_n$, and that
\begin{equation}
 \label{CO10}
\forall n\in \mathbb N, \ \left\| \Pi_{L^2_{R_n}}^{\bot}u_{1,n}\right\|_{L^2_{R_n}}=1, \,\mbox{and}\;\lim_{n\to\infty} \left( \lim_{t\to\infty}\int_{|x|>R_n+|t|} |\nabla_{t,x}u_{L,n}(t,x)|^2dx \right)=0.
\end{equation} 
We decompose $u_{1,n}$ as follows.
\begin{equation}
 \label{CO11}
 u_{1,n}=c_n\Lambda W+v_{1,n},\; r>R_n\quad \text{where} \int_{|x|>R_n}v_{1,n}\Lambda W=0.
\end{equation}
We extend $v_{1,n}$ by $v_{1,n}(r)=0$ for $0<r<R_n$. We let $v_{L,n}$ be the solution of \eqref{LW} with initial data $(0,v_{1,n})$, and note that $v_{L,n}(t,r)=u_{L,n}(t,r)-c_nt\Lambda W$ for $r>R_n+|t|$.

Extracting subsequences, we can assume that we are in one of the three following cases:
\begin{align*}
 \mbox{Case 1:} \quad& \forall n,\;R_n=0\quad \text{or}\quad \forall n,\;R_n\geq 1\\
\mbox{Case 2:} \quad& \forall n,\;R_n\in (0,1)\quad \text{and}\quad\lim_{n\to\infty} R_n=R_{\infty}\in (0,1]\\
\mbox{Case 3:} \quad& \forall n,\,R_n\in (0,1) \quad \text{and}\quad \lim_{n\to\infty} R_n=0.
\end{align*}

\medskip

\noindent\textbf{Case 1.} We assume $R_n=0$ or for all $n$, $R_n\geq 1$. Extracting subsequences, we can also assume $\lim_{n\to\infty} R_n=R_{\infty}\in \{0\}\cup [1,\infty]$. We note that in this case, $v_{1,n}=\Pi^{\bot}_{L^2_{R_n}}u_{1,n}$, so that $\|v_{1,n}\|_{L^2_{R_n}}=1$ by \eqref{CO10}. We first claim
\begin{equation}
\label{CO12}
 v_{1,n}\xrightharpoonup[n\to\infty]{} 0\text{ in }L^2(\RR^6).
\end{equation} 
If $R_{\infty}=\infty$, this is obvious since $v_{1,n}(r)=0$ for $r\in (0,R_{n})$. If not, we introduce (after extraction) the weak limit $v_1$ of $v_{1,n}$ in $L^2$. We will prove that $v_1=0$. By weak convergence and the orthogonality condition in \eqref{CO11}, we have 
\begin{equation}
 \label{CO20}
 \int_{|x|>R_{\infty}} v_1\Lambda W=0.
\end{equation} 
We denote by $v_L$ the solution of \eqref{LW} with initial data $(0,v_1)$. Fix now any $R>R_\infty$, and define $v_{1,n}^R=\indic_{\{r\geq R\}}v_{1,n}$ and $v_{1}^R=\indic_{\{r\geq R\}}v_{1}$, and $v_{L,n}^R$ and $v_L^R$ the solutions to \eqref{LW} with initial data $(0,v_{1,n}^R)$ and $(0,v_{1}^R)$ respectively. Then, $v_{L,n}^R$ coincides with $u_{L,n}^R-c_nt\Lambda W$ for $r\geq R+|t|$ by finite speed of propagation, so that using \eqref{CO10}:
\begin{equation} \label{bd:interproofchannel1}
\lim_{n\to \infty} \lim_{t\to \infty} \int_{r\geq R+|t|} |\nabla v_{L,n}^R|^2+(\pa_tv_{L,n}^R)^2=0.
\end{equation}
We have that $v_1^R$ is the weak limit of $(v_{1,n}^R)_n$, so that $v_1^R$ belongs to the closure in $L^2$ of the convex hull of $\{v_{1,n}^R\}_{n\geq N}$ for any $N$. Hence, using energy estimates, $v_{L}^R $ belongs to the closure in $L^{\infty}_t (\mathcal H_{r\geq |t|})$ of the convex hull of $\{ v_{L,n}^R\}_{n\geq N}$ for any $N$. This fact and the estimate \eqref{bd:interproofchannel1} then imply:
\begin{equation*}
 \lim_{t\to \infty} \int_{r\geq R+|t|} |\nabla v_{L}^R|^2+(\pa_tv_{L}^R)^2=0.
\end{equation*}
Thus, Lemma \ref{L:rigiditylinearoddintime} implies that $v_1^R(r)=c_R\Lambda W(r)$ for all $r\geq R$ for some $c_R\in \mathbb R$. Coming back to the definition of $v_1^R$, we see that $c_R=c$ is independent of $R>R_\infty$ so that $v_1(r)=c\Lambda W(r)$ for $r>R_\infty$, and that $c=0$ using \eqref{CO20}. Hence \eqref{CO12}.

Next, we let $v_{F,n}$ be the solution of the free wave equation $(\partial_t^2-\Delta)v_{F,n}=0$ with initial data $(0,v_{1,n})$. By \eqref{CO12} and \cite[Lemma 3.7]{DuKeMe20}, we have 
$$\sup_{t\in \RR} \left\| \vec{v}_{F,n}(t)-\vec{v}_{L,n}(t)\right\|_{\HHH_{|t|}}\underset{n\to\infty}{\longrightarrow}0.$$
Thus by \eqref{CO10} and \eqref{CO11},
\begin{equation}
 \label{CO21}
 \lim_{n\to\infty}\lim_{t\to\infty} \int_{|x|>R_n+|t|} \left|\nabla_{t,x}v_{F,n}(t,x)\right|^2dx=0.
\end{equation} 
On the other hand, by the result \eqref{free_exterior} proved in \cite{CoKeSc14} if $R_n=0$ or Proposition \ref{P:CK1.4} if $R_n>0$, we have
$$\lim_{t\to\infty} \int_{|x|>R_n+|t|} |\nabla_{t,x}v_{F,n}(t,x)|^2dx\gtrsim
\begin{cases}
\|v_{1,n}\|^2_{L^2_{R_n}}&\text{ if }R_n=0\\
\|v_{1,n}\|^2_{L^2_{R_n}}-\left(\left\|\frac{1}{r^4}\right\|_{L^2_{R_n}}^{-1} \int_{|x|>R_n}\frac{1}{r^4}v_{1,n}(x)dx\right)^2&\text{ if }R_n>0.
\end{cases}
$$
Recall that $\|v_{1,n}\|^2_{L^2_{R_n}}=1$ for all $n$ by \eqref{CO10}.
In the case where $R_n=0$ for all $n$, we see that it contradicts \eqref{CO21}. When $R_n>0$ for all $n$, since $v_{1,n}$ converges weakly to $0$ in $L^2$, and $\lim_{n\to\infty} R_n>0$, we obtain 
$$ \lim_{n\to\infty}\left\|\frac{1}{r^4}\right\|_{L^2_{R_n}}^{-1} \int_{|x|>R_n}\frac{1}{r^4}v_{1,n}(x)dx=0,$$
and the contradiction follows in this case also.

\medskip

\textbf{Case 2.} We now assume that $R_n\in (0,1)$ for all $n$, and that $\lim R_n=R_{\infty}\in (0,1]$. In this case, by the equality $u_{1,n}=c_n\Lambda W+v_{1,n}$, we have $\Pi^{\bot}_{L^2_{R_n}}u_{1,n}=\Pi^{\bot}_{L^2_{R_n}}v_{1,n}$, and thus, by \eqref{CO10}, $\|v_{1,n}\|_{L^2_{R_n}}\geq 1$. Letting 
$$\tilde{u}_{1,n}=\frac{1}{\|v_{1,n}\|_{L^2_{R_n}}}u_{1,n}, \quad \tilde{v}_{1,n}=\frac{1}{\|v_{1,n}\|_{L^2_{R_n}}}v_{1,n},$$
and $\tilde{u}_{L,n}$ the solution of \eqref{LW} with initial data $(0,\tilde{u}_{1,n})$, we see that $\tilde{u}_{1,n}=\tilde{c}_n\Lambda W+\tilde{v}_{1,n}$, $\int_{|x|>R_n} \tilde{v}_{1,n}\Lambda W=0$, $\|\tilde{v}_{1,n}\|_{L^2_{R_n}}=1$, $\lim_{n\to\infty} R_n=R_{\infty}>0$ and 
$$ \lim_{n\to\infty}\lim_{t\to\infty}\int_{|x|>R_n+|t|} |\nabla_{t,x}\tilde{u}_{L,n}(t,x)|^2dx=0.$$
With exactly the same proof as in Case 1, replacing $u_{L,n}$ by $\tilde{u}_{L,n}$, one obtains a contradiction.

\medskip

\textbf{Case 3.} We assume that $R_n\in (0,1]$ for all $n$ and that $\lim_{n\to\infty} R_n=0$. 

\noindent{\textbf{Step 1.}}

We let $\widetilde{\Gamma}_L$ be the solution of \eqref{LW} with initial data $(0,\widetilde{\Gamma})$. Since $\widetilde{\Gamma}$ is in $L^2_R$ for all $R>0$, this solution is well-defined (by finite speed of propagation and since \eqref{LW} is well-posed in $\HHH$) for $r>|t|+R$, and has finite energy there when $R>0$ is fixed. We claim that there exists a constant $C>0$ such that for all $R\in (0,1)$,

\begin{equation}
 \label{CO40}
 \frac{1}{C}\leq \lim_{t\to+\infty} \int_{|x|>R+|t|}|\nabla_{t,x}\widetilde{\Gamma}_L(t,x)|^2dx\leq C.
\end{equation} 
Indeed, $\widetilde{\Gamma}\neq \Lambda W$, thus by Lemma \ref{L:rigiditylinearoddintime},
$$\lim_{t\to+\infty} \int_{|x|>1+|t|}|\nabla_{t,x}\widetilde{\Gamma}_L(t,x)|^2dx>0,$$
and the left-hand side inequality in \eqref{CO40} follows. Furthermore, denoting by $\Xi_L=\widetilde{\Gamma_L}-t\widetilde{\Gamma}$, we have

\begin{equation} \label{id:defXiL}
 \left\{
 \begin{aligned}
 (\partial_t^2-\Delta +V) \Xi_L&=t(-\Delta +V)\widetilde{\Gamma}\\
 \vec{\Xi}_L&=(0,0),
 \end{aligned}
 \right.
\end{equation} 

and thus, since $(-\Delta+V)\widetilde{\Gamma}=0$ for $r\leq 10$ and for $r\geq 12$ from \eqref{id:deftildeGamma}, we see by simple energy estimates that 
$$\lim_{t\to\infty} \int_{|x|>|t|} |\nabla_{t,x}\Xi_L(t,x)|^2dx$$
is finite, which yields the right-hand side inequality in \eqref{CO40}.

\noindent \textbf{Step 2.}

Next, we expand $u_{1,n}$ as follows:
$$ u_{1,n}=d_n\Lambda W+\tilde{d}_n\widetilde{\Gamma} +\Pi^{\bot}_{L^2_{R_n}}u_{1,n}, \quad r>R_n$$
and by \eqref{CO10} and the lower bound in \eqref{CO40}, we deduce that  $\tilde{d}_n$ is bounded. Using the definition \eqref{CO11} of $v_{1,n}$, we obtain 
\begin{equation}
\label{CO41}
(c_n-d_n)\Lambda W=\tilde{d}_n\widetilde{\Gamma}+\Pi^{\bot}_{L^2_{R_n}}u_{1,n}-v_{1,n}, 
\end{equation} 
and thus, taking the scalar product with $\Lambda W$, $(c_n-d_n)\int_{|x|>R_n}|\Lambda W|^2=\tilde{d}_n\int_{R_n}\Lambda W\widetilde{\Gamma}$. Since $\Lambda W \widetilde{\Gamma}\in L^1(\RR^6)$, we deduce that $c_n-d_n$ is also bounded. As a consequence, and since $\widetilde{\Gamma}\in L^2\left(r^{7.5}\langle r \rangle^{-2.5}dr\right)$, we obtain that $v_{1,n}$ is bounded in $L^2\left(r^{7.5}\langle r\rangle^{-2.5}dr\right)$. With the same argument as in Case 1, using Lemma \ref{L:rigiditylinearoddintime}, we see that $v_{1,n}$ converges weakly to $0$ in $L^2_R$ for all $R>0$, and thus 
$$ v_{1,n}\xrightharpoonup[n\to\infty]{} 0,\quad \text{weakly in }L^2\left( r^{7.5}\langle r\rangle^{-2.5}dr \right).$$
 
\noindent\textbf{Step 3.}

Next, we decompose $v_{1,n}$ as follows for $r>R_n$
$$ v_{1,n}=\tilde{c}_n\widetilde{\Gamma} +w_{1,n},\quad \int_{|x|>R_n} w_{1,n}\widetilde{\Gamma}=0,$$
and let $w_{1,n}(r)=0$ for $r\in (0,R_n]$. We claim 
\begin{equation}
 \label{CO50}
 \lim_{n\to\infty}\tilde{c}_n=0,\quad w_{1,n}\xrightharpoonup[n\to\infty]{}0\text{ in }L^2.
\end{equation} 
Indeed, \eqref{CO41} implies $(\tilde{c}_n-\tilde{d}_n)\widetilde{\Gamma}+w_{1,n}=(d_n-c_n)\Lambda W+\Pi^{\bot}_{L^2_{R_n}} u_{1,n}$. Thus 
 \begin{equation}
 \label{CO60}
 \|w_{1,n}\|^2_{L^2_{R_n}}+\left( \tilde{c}_n-\tilde{d}_n \right)^2\left\|\widetilde{\Gamma}\right\|^2_{L^2_{R_n}}=\left\|\Pi^{\bot}_{L^2_{R_n}}u_{1,n}\right\|^2_{L^2_{R_n}}+(c_n-d_n)^2\left\|\Lambda W\right\|^2_{L^2_{R_n}}. 
 \end{equation} 
Since the right-hand side of \eqref{CO60} is bounded, we obtain that its left-hand side is also bounded and thus that $w_{1,n}$ is bounded in $L^2$, and that $\tilde{c}_n-\tilde{d}_n$ is bounded. Since we have proved above that $\tilde{d}_n$ is bounded, we obtain that $\tilde{c}_n$ is bounded. Extracting subsequences, we can assume
$$\lim_{n}\tilde{c}_n=\tilde{c},\quad w_{1,n}\xrightharpoonup[n\to\infty]{} w_1\text{ in }L^2.$$
Since $v_{1,n}\xrightharpoonup[n\to\infty]{}0$ in $L^2( r^{7.5}\langle r\rangle^{-2.5}dr)$, we obtain $0=\tilde{c}\widetilde{\Gamma}+w_1$, and thus, since $w_1\in L^2$ and $\widetilde{\Gamma}\notin L^2$, $\tilde{c}=0$ and $w_1=0$. Hence \eqref{CO50}.

\noindent\textbf{Step 4.}

Let $w_{L,n}$ be the solution of \eqref{LW} with initial data $(0,w_{1,n})$. From $v_{1,n}=\tilde{c}_n\widetilde{\Gamma}+w_{1,n}$, \eqref{CO40} and \eqref{CO50}, we obtain 
\begin{equation}
 \label{CO61}
 \lim_{n\to\infty}\left( \lim_{t\to\infty}\int_{|x|>R_n+|t|}|\nabla_{t,x}w_{L,n}(t,x)|^2dx \right)=0.
\end{equation} 
Let $w_{F,n}$ be the solution of the free wave equation with initial data $(0,w_{1,n})$. Since $w_{1,n}\xrightharpoonup{} 0$ in $L^2$, we have 
$$ \lim_{n\to\infty} \sup_{t\in \RR} \left\|w_{L,n}(t)-w_{F,n}(t)\right\|_{\HHH_{|t|}}=0.$$
By \eqref{CO61},
\begin{equation}
 \label{CO70}
 \lim_{n\to\infty}\left( \lim_{t\to\infty}\int_{|x|>R_n+|t|}|\nabla_{t,x}w_{F,n}(t,x)|^2dx \right)=0.
\end{equation} 
Furthermore 
\begin{equation} \label{bd:interproofchannel3}
 \| \frac{1}{r^4}\|_{L^2_{R_n}}^{-1} \int_{r\geq R_n} \pa_t w_{F,n}(0)\frac{1}{r^4}r^5 dr \rightarrow 0.
 \end{equation}
Indeed,
$$
\left| \int_{r\geq R_n} \pa_t w_{F,n}(0)\frac{1}{r^4}r^5 dr\right| = \left| \int_{r\geq R_n} \pa_t w_{F,n}(0)(\frac{1}{r^4}-\frac{1}{c}\tGamma)r^5 dr \right|\leq C\| \frac{1}{r^4}-\frac{1}{c}\tGamma \|_{L^2_{R_n}},
$$
and hence \eqref{bd:interproofchannel3} using \eqref{id:asymptoticGamma}. By \eqref{CO70}, \eqref{bd:interproofchannel3} and Proposition \ref{P:CK1.4}, we have 
$$\lim_{n\to\infty}\|w_{1,n}\|_{L^2_{R_n}}=0.$$
However, 
$$ \|w_{1,n}\|_{L^2_{R_n}}\geq \|\Pi^{\bot}_{L^2_{R_n}}w_{1,n}\|_{L^2_{R_n}}=\|\Pi^{\bot}_{L^2_{R_n}}u_{1,n}\|_{L^2_{R_n}}=1,$$
which gives a contradiction.
\end{proof}

\subsection{Channels of energy around the ground state for even in time solutions}

We prove in this Subsection:

\begin{lemma} \label{L:channelsevenunsoliton}

There exists $C>0$ such that the following holds true. Let $u_L$ be a radial solution of \eqref{LW}. Then:
$$
\| \nabla \Pi_{\dot H^1}^\perp  u_0 \|_{Z_{-3}}^2\leq C \sum_{\pm } \lim_{t\rightarrow \pm \infty}  \int_{r\geq |t|} |\nabla_{t,x}u_L(t,x)|^2dx.
$$

\end{lemma}

Lemma \ref{L:channelsevenunsoliton} is a consequence of a more general estimate for the set $\{r\geq R+|t|\}$ for any $R>0$ stated in Lemma \ref{L:channelsevenunsoliton2} below. We introduce an element of the generalized kernel of $-\Delta+V$.

\begin{lemma}

There exists a solution $\Upsilon$ to $(-\Delta +V)\Upsilon =\Lambda W$ such that, for two constants $c_0\neq 0$ and $c_\infty \neq 0$:
\begin{gather} \label{bd:Upsilon}
\Upsilon(r)\underset{r\to 0}{=} c_0r^{-4}+O(r^{-2}) , \qquad \Upsilon(r)\underset{r\to \infty}{=} c_\infty r^{-2}+O(r^{-4}\log r)\\
\label{bd:Upsilon'}
\frac{\partial \Upsilon}{\partial r}(r)\underset{r\to 0}{=} -4c_0r^{-5}+O(r^{-3}) , \qquad \frac{\partial \Upsilon}{\partial r}(r)\underset{r\to \infty}{=} -2c_\infty r^{-3}+O(r^{-5}\log r).
\end{gather}

\end{lemma}

\begin{proof}

Solving $(-\Delta +V)\Upsilon=\Lambda W$, we define $\Upsilon$ by:
$$
\Upsilon(r)=-\Lambda W(r)\int_0^r \Lambda W (s)\Gamma (s)s^5 ds-\Gamma (r)\int_r^\infty \Lambda W^2 (s) s^5 ds.
$$
Using  that, by the definition of $\Lambda W$, for some $c\neq 0$, $d\neq 0$, 
\begin{align*}
\Lambda W(r)&=c+O(r^{2}), &\frac{d \Lambda W}{dr}(r) &=O(r) &&r\to 0, \\ 
\Lambda W(r)&=d r^{-4}+O(r^{-6}),&\frac{d\Lambda W}{dr}(r)&=-4dr^{-5}+O(r^{-7})&&r\to\infty, 
\end{align*}
and the asymptotics behaviours of $\Gamma$ and $\frac{d \Gamma}{dr}$ given by \eqref{id:asymptoticGamma}, \eqref{id:asymptoticparGamma}, 
one obtains the desired asymptotic behaviour \eqref{bd:Upsilon} by direct computations.
%and , and $\Lambda W(r)=dr^{-4}+O(r^{-6})$, $\Lambda W(r)= -4dr^{-5}+O(r^{-7}as $r\to \infty$, with $(c,c',d,d')\in \left(\mathbb{R}\setminus\{0\}\right)^{4}$, 

\end{proof}

Observe that the function $(t,r)\mapsto \Upsilon(r)-\frac{t^2}{2}\Lambda W(r)$ solves \eqref{LW} but $\vec{\Upsilon}(t)$ fails critically to belong to the spaces $\HHH_R$, $R>0$. This is thus a resonance for Equation \eqref{LW}.

We define for $\alpha \in \mathbb R$ and $R>0$:
$$
\| f\|_{Z_{\alpha, R}}=\sup_{\rho\geq R} \  \frac{\rho^{-3-\alpha}}{\langle \log \frac{\rho}{\langle R\rangle} \rangle} \ \| f\|_{L^2(\rho \leq r\leq 2\rho)}, \qquad \| f\|_{\widetilde{Z}_{\alpha,R}}=\sup_{\rho>R} \  \rho^{-3-\alpha} \| f\|_{L^2(\rho\leq r\leq 2\rho)}
$$
and for any set $E\subset Z_{\alpha, R}$: 
$$
d_{Z_{\alpha, R}}(u,E)=\inf_{v\in E} \| u-v\|_{Z_{\alpha, R}}.
$$
Then we claim:

\begin{lemma} \label{L:channelsevenunsoliton2}

There exists $C>0$ such that, for any radial solution $u_L$ of \eqref{LW} with $\vec u(0)\in \mathcal H$, for any $R>0$:
$$
d_{Z_{-3,24R}}(\pa_r u_0,\text{Span}(\pa_r \Lambda W,\pa_r \Upsilon))^2 \leq C \sum_{\pm } \lim_{t\rightarrow \pm \infty}  \int_{r\geq |t|+R} |\nabla_{t,x}u(t,x)|^2dx.
$$

\end{lemma}

\begin{remark}

The constant $24$ is arbitrary, and used to ease the proof, but the result could be proved for any constant $>1$.

\end{remark}

Assuming Lemma \ref{L:channelsevenunsoliton2} we can prove Lemma \ref{L:channelsevenunsoliton}.

\begin{proof}[Proof of Lemma \ref{L:channelsevenunsoliton}]

Let $R_n\downarrow 0$. Then for any $n$, applying Lemma \ref{L:channelsevenunsoliton2} with $R=R_n/24$, we get that two constants $c_n$ and $c_n'$ exist such that
$$
u_0=c_n \Lambda W+c_n'\Upsilon +\tilde u_{0,n}, \qquad \mbox{with}\qquad  \| \pa_r \tilde u_{0,n}\|_{Z_{-3,R_n}}^2 \lesssim  \sum_{\pm } \lim_{t\rightarrow \pm \infty}  \int_{r\geq |t|} |\nabla_{t,x}u(t,x)|^2dx .
$$
Hence, in the space $L^2(\{1<r<2\})$, both $\pa_r u_0$ and $\pa_r \tilde u_{0,n}$ are uniformly bounded, and so is $c_n\pa_r \Lambda W+c_n'\pa_r \Upsilon $. As $\pa_r \Lambda W$ and $\pa_r \Upsilon $ are non collinear in $L^2(\{1<r<2\})$, we get that $c_n$ and $c_n'$ are bounded sequences. Therefore, $c_n'\pa_r \Upsilon$ is uniformly bounded in $Z_{-3,R_n}$, implying $c_n'\to 0$ because of \eqref{bd:Upsilon'}. 
Up to extracting a subsequence, $c_n$ converges to some limit $c_\infty$, and the uniform bound on $\tilde u_{0,n}$ then implies that at fixed $R>0$, 
$$\langle \log R\rangle ^{-2}\| \pa_r u_0-c_\infty \pa_r\Lambda W \|_{L^2(\{R<r<2R\})}^2 \lesssim  \sum_{\pm } \lim_{t\rightarrow \pm \infty}  \int_{r\geq |t|} |\nabla_{t,x}u(t,x)|^2dx,$$
and thus
$$ \| \pa_r u_0-c_\infty \pa_r\Lambda W \|_{Z_{-3}}^2 \lesssim  \sum_{\pm } \lim_{t\rightarrow \pm \infty}  \int_{r\geq |t|} |\nabla_{t,x}u(t,x)|^2dx.$$
This implies Lemma \ref{L:channelsevenunsoliton}, upon noticing that 
$$\left\|\nabla \big(\Pi_{\dot{H}^1}^{\bot} u_0\big)\right\|^2_{Z_{-3}}=\left\|\nabla \big(\Pi_{\dot{H}^1}^{\bot} (u_0-c_{\infty}\Lambda W\big)\right\|^2_{Z_{-3}}\lesssim \left\|\nabla (u_0-c_{\infty}\Lambda W)\right\|^2_{Z_{-3}},$$
where the last bound follows from explicit computation, using that $|\Lambda W|+\left|r\frac{d}{dr}\Lambda W\right|\lesssim \langle r\rangle^{-4}$.
\end{proof}

The rest of this Subsection is devoted to the proof of Lemma \ref{L:channelsevenunsoliton2}. We decompose its proof in three intermediate lemmas and start with an elliptic result.

\begin{lemma}[Elliptic estimate close to a soliton] \label{L:ellipticunsoliton}

There exists $C>0$ such that the following holds true. Let $R>0$ and assume that $u \in Z_{-2,R}(\mathbb R^6)$ solves for $r\geq R$:
\begin{equation}\label{eq:elliptic1}
-\Delta u+V u=f+\pa_r g
\end{equation}
where $u$, $f$ and $g$ are radial and satisfy:
\begin{equation}\label{eq:hpfg}
\| f\|_{Z_{-4,R}}+\| g\|_{Z_{-3,R}}<\infty.
\end{equation}
Then, if $R\geq 1$:
\begin{equation}\label{bd:ellipticdecay1}
d_{_{Z_{-2, R}}}( u, \textup{Span}(\Lambda W)) \leq C\left(\| f\|_{Z_{-4,R}}+\| g\|_{Z_{-3,R}}  \right)
\end{equation}
while if $0<R< 1$:
\begin{equation}\label{bd:ellipticdecay1'}
d_{_{Z_{-2, R}}}( u, \textup{Span}(\Lambda W,\Upsilon)) \leq C\left(\| f\|_{Z_{-4,R}}+\| g\|_{Z_{-3,R}}  \right).
\end{equation}

\end{lemma}

\begin{proof}

We can assume without loss of generality that $f=0$. This is because, defining $\tilde g=g+\tilde f$ with $\tilde f(r)=-\int_r^\infty  f(s)ds$, then $f=\pa_r \tilde f$ so that $-\Delta u+V u=\pa_r \tilde g$, and moreover $|\tilde f(r)|\lesssim \langle \log \frac{r}{\langle R\rangle}\rangle r^{-3}\| f\|_{Z_{-4,R}}$ for all $r>0$ by elementary estimates, so that $\| \tilde g\|_{Z_{-3,R}}\lesssim \| f\|_{Z_{-4,R}}+\| g \|_{Z_{-3,R}}$.\\

\noindent \textbf{Step 1}. \emph{The case $0<R<1$}. Note that in this case:
\be \label{unsoliton:id:equivnorms1}
\| \cdot \|_{Z_{\alpha, R}}\approx \sup_{\rho\geq R} \  \frac{\rho^{-3-\alpha}}{\langle \log \rho\rangle} \ \| \cdot \|_{L^2(\rho\leq r\leq 2\rho)}.
\ee
Solving \eqref{eq:elliptic1} using \eqref{id:LambdaWGammawronskian}, and $u\in Z_{-2,R}$, we find that there exists $c\in \mathbb R$ such that for all $r>R$:
$$
 u(r)
 =c\Lambda W(r)+\Lambda W(r)\int_{1}^r g (s)\pa_s (\Gamma (s)s^5)ds+\Gamma (r) \int_r^\infty g (s)\pa_s (\Lambda W (s)s^5)ds.
$$
We recall that, using \eqref{bd:Upsilon}, there exists $\tilde c\neq 0$ such that
\be \label{unsoliton:id:UpsilonGamma}
\tilde c \Upsilon (r)=\Gamma (r)+O(r^{-2}) \qquad \mbox{as }r\to 0.
\ee
Introducing $\bar c= \tilde c \int_R^\infty g\pa_s (\Lambda Ws^5)ds$, we decompose $u=c\Lambda W+\bar c \Upsilon+\tilde u$ with:
\begin{align*}
\tilde u(r)=\Lambda W(r)\int_{1}^r g (s)\pa_s (\Gamma (s)s^5)ds+\Gamma (r) \int_r^\infty g (s)\pa_s (\Lambda W (s)s^5)ds-\bar c \Upsilon(r).
\end{align*}
For $r\geq 1$ let $k(r)\in \mathbb N$ such that $2^{k(r)-1}\leq r<2^{k(r)}$. Using dyadic partitioning, Cauchy-Schwarz, \eqref{id:asymptoticGamma}, \eqref{id:asymptoticparGamma} and \eqref{unsoliton:id:equivnorms1} we estimate that:
\begin{align}
\nonumber \left| \int_{1}^r g (s)\pa_s (\Gamma (s)s^5)ds \right| & \leq \sum_{k=1}^{k(r)-1} \int_{2^k}^{2^{k+1}} |g (s)\pa_s (\Gamma (s)s^5)|ds  \\
\nonumber & \leq  \sum_{k=1}^{k(r)-1} \left( \int_{2^{k}}^{2^{k+1}} g^2 (s)s^5 ds\right)^{\frac 12}  \left( \int_{2^{k}}^{2^{k+1}} s^3 ds\right)^{\frac 12} \\
\label{unsoliton:bd:ellipticinter} &\leq  \sum_{k=1}^{k(r)-1} \langle k \rangle \| g\|_{Z_{-3,R}} 2^{2k} \ \leq \ r^{2} \langle \log r \rangle \| g\|_{Z_{-3,R}}.
\end{align}
Hence $|\Lambda W(r)\int_{1}^r g (s)\pa_s (\Gamma (s)s^5)ds|\lesssim r^{-2} \langle \log r \rangle \| g\|_{Z_{-3,R}}$. Similarly, $|\Gamma (r) \int_r^\infty g (s)\pa_s (\Lambda W (s)s^5)ds | \lesssim  r^{-2} \langle \log r \rangle \| g\|_{Z_{-3,R}}$, and hence using $|\bar c|\lesssim \| g\|_{Z_{-3,R}}$ and \eqref{bd:Upsilon}:
\begin{equation}\label{id:estimationellipticinter1}
\forall r\geq 1, \qquad \tilde u(r)=O(r^{-2}\langle \log r\rangle \| g\|_{Z_{-3,R}}).
\end{equation}
Next, for $R\leq r\leq 1$ we decompose:
$$
\Gamma (r) \int_r^\infty g (s)\pa_s \left(\Lambda W (s)s^5\right)ds-\bar c \Upsilon(r)=-\Gamma (r) \int_R^r g (s)\pa_s \left(\Lambda W (s)s^5\right)ds+\bar c\left(\frac{1}{\tilde c} \Gamma(r) -\Upsilon(r)\right). 
$$
With computations similar to \eqref{unsoliton:bd:ellipticinter}, we obtain $\left|\Gamma (r) \int_R^r g (s)\pa_s (\Lambda W (s)s^5)ds\right|\lesssim r^{-2}\langle \log r \rangle \|g\|_{Z_{-3,R}}$ and $\left|\Lambda W(r)\int_{1}^r g (s)\pa_s (\Gamma (s)s^5)ds\right|\lesssim r^{-2}\langle \log r \rangle \|g\|_{Z_{-3,R}} $. Using this, \eqref{unsoliton:id:UpsilonGamma} and $|\bar c|\lesssim \|g\|_{Z_{-3,R}}$ we infer:
\begin{equation}\label{id:estimationellipticinter2}
\forall r\in [R,1], \qquad \tilde u(r)=O(r^{-2}\langle \log r\rangle \| g\|_{Z_{-3,R}}).
\end{equation}
Combining \eqref{id:estimationellipticinter1} and \eqref{id:estimationellipticinter2} shows $\| \tilde u\|_{Z_{-2,R}}\lesssim \| g\|_{Z_{-3,R}}$, and hence \eqref{bd:ellipticdecay1'} holds true and the Lemma is proved in the case $0<R<1$.\\

\noindent \textbf{Step 2}. \emph{The case $R\geq 1$}. Note that in this case:
\be \label{unsoliton:id:equivnorms2}
\| \cdot \|_{Z_{\alpha, R}}\approx \sup_{\rho\geq R} \  \frac{\rho^{-3-\alpha}}{\langle \log \frac{\rho}{R}\rangle} \ \| \cdot \|_{L^2(\rho\leq r\leq 2\rho)}.
\ee
We solve \eqref{eq:elliptic1} and get that for some $c\in \mathbb R$, for all $r>R$:
$$
 u(r)
 =c\Lambda W(r)+\Lambda W(r)\int_{R}^r g (s)\pa_s (\Gamma (s)s^5)ds+\Gamma (r) \int_r^\infty g (s)\pa_s (\Lambda W (s)s^5)ds.
$$
For $r\geq R$ let $k(r)\in \mathbb N$ such that $2^{k(r)-1}R \leq r<2^{k(r)}R$. Using dyadic partitioning, Cauchy-Schwarz, \eqref{id:asymptoticGamma}, \eqref{id:asymptoticparGamma} and \eqref{unsoliton:id:equivnorms2} we estimate that:
\begin{align}
\nonumber \left| \int_{R}^r g (s)\pa_s (\Gamma (s)s^5)ds \right| & \leq \sum_{k=1}^{k(r)-1} \int_{2^kR}^{2^{k+1}R} |g (s)\pa_s (\Gamma (s)s^5)|ds  \\
\nonumber &\lesssim  \sum_{k=1}^{k(r)-1} \langle k \rangle \| g\|_{Z_{-3,R}} 2^{2k}R^2 \ \lesssim \ r^{2} \langle \log \frac rR \rangle \| g\|_{Z_{-3,R}}.
\end{align}
Similarly, $|\Gamma (r) \int_r^\infty g (s)\pa_s (\Lambda W (s)s^5)ds|\lesssim  r^{-2} \langle \log \frac rR \rangle \| g\|_{Z_{-3,R}}$. Combining, this proves \eqref{bd:ellipticdecay1} and the Lemma is proved in the case $R\geq 1$.

\end{proof}

If $u_L$ is a solution of \eqref{LW} with initial data in $\HHH$, we denote:
$$
\| \partial_tu_L\|_{Y_R}=\sup_{\substack{t\in \mathbb R\\ \rho> |t|+R}} \ \left\| \Pi^\perp_{L^2_{\rho}} \partial_tu_L(t)\right\|_{L^2_{\rho}},
$$
where $\Pi^\perp_{L^2_{\rho}}$ is defined in \eqref{id:defpiL2R1}. For $k\in \mathbb Z$, $R_k=2^k$ and $0\leq t\leq R_k$, we write for the remaining part of this Subsection:
\begin{equation}\label{id:decompositionutY}
\pa_t u_L(t)= \begin{cases}\alpha_{k}(t)\Lambda W+g_k(t) & \mbox{for }R_k\geq 1,\; t\in \RR\\ \alpha_k(t)\Lambda W+\tilde \alpha_k (t)\tilde \Gamma+g_k(t) & \mbox{for }R_k< 1, \;t\in \RR.  \end{cases} 
\end{equation}
where $g_k(t)=\Pi^\perp_{L^2_{R_k}}\partial_tu_L(t)$, so that:
\begin{equation}\label{bd:ginterufromut}
\sup_{\substack{k\in \mathbb Z, \ R_k\geq R\\ |t|\leq R_k-R}}\| g_k(t)\|_{L^2_{R_k}}\leq \| \partial_tu_L\|_{Y_R}.
\end{equation} 

\begin{lemma}[Estimating $u_L$ in weighted $L^2$ from $\pa_t u_L$ for waves around a soliton] \label{L:estimateufromutunsoliton}
The\-re exists $C>0$ such that the following holds. Assume that $u_L$ is a radial solution to \eqref{LW} with $\vec u(0)\in \mathcal H$. Then, for any $R>0$,
\be \label{unsoliton:bd:evenR>0}
d_{Z_{-2,6R}}(u_0,\text{Span}(\Lambda W,\Upsilon)) \leq C  \| \partial_tu_L\|_{Y_R}.
\ee
\end{lemma}

\begin{proof}

Note first that it suffices to show the result for $u$ even in time. Indeed, for general $u$, it suffices to apply the result to $u_+(t)=\frac{1}{2}(u(t)+u(-t))$, noticing that $\pa_t u_+(t)=\frac{1}{2}(\pa_t u(t)-\pa_t u(-t))$ satisfies $\| \pa_t u_+\|_Y\leq \| \pa_t u\|_Y $.\\

\noindent \textbf{Step 1}. \emph{An elliptic equation for $u_0$}. Let for $k\in \mathbb Z$, $R_k=2^k$ and $t_k=\frac{R_k}{2}$. Let $k_0$ be the unique integer such that $R_{k_0-2}\leq R<R_{k_0-1}$. Note that for all $k\geq k_0$, there holds $R_k\geq R+t_k$. Let $\chi$ be a smooth cut-off function that is supported in $[1,3]$ such that for all $r>0$, $\sum_{k\in \mathbb Z} \chi (r/R_k)=1$. We write $\chi_k(r)=\chi(r/R_k)$. Let $\tilde \chi$ be a nonzero, smooth and nonnegative function with support inside $[2,3]$.

There exists a family of smooth functions $\Phi_k$ for $k\geq -1 $, and $\tilde \Phi_k$ for $k\leq -1$ supported in $[2R_k,3R_k]$, such that $|\pa_r^j \Phi_k|\lesssim  R_k^{-j}$ and $|\pa_r^j \tilde \Phi_k|\lesssim R_k^{-j}$ for $j=0,1,2$, and such that for all $k\leq -1$, $\int \Gamma \tilde \Phi_k =R_k^2$ and $\int \Lambda W \tilde \Phi_k=0$, and for all $k\geq -1$, $\int \Lambda W \Phi_k =R_k^2$, and $\int \Gamma \Phi_{-1} =0$. The proof of this fact, a direct consequence of \eqref{id:asymptoticGamma}, and from the fact that $\Lambda W(r)\approx 1/r^4$ for large $r$, is omitted. Note that we do not need to require $\int \Gamma \Phi_k=0$ for $k\geq 0$. \\

Fix $k\in \mathbb Z$. Applying the fundamental Theorem of Calculus for $\pa_t u_L$ and then $u_L$ between $t=0$ and $t=R_k$, using \eqref{LW} and $\pa_t u(0)=0$ as $u$ is even in time, gives:
\begin{equation}\label{eq:ellipticinter1unsoliton}
-\Delta u_0+V u_0=- \frac{1}{t_k}\pa_t u_L(t_k)-(-\Delta +V)\left( \int_0^{t_k}(1-\frac{t}{t_k})\pa_t u_L(t)dt\right).
\end{equation}
Let $k\geq k_0$. It follows from \eqref{id:LambdaWGammazeroes} and \eqref{id:decompositionutY} that for all $r\geq R_k$ if $k\geq 0$, and for all $R_k\leq r\leq 2$ if $k\leq -1$:
$$
(-\Delta +V)\left( \int_0^{t_k}(1-\frac{t}{t_k})\pa_t u_L(t)dt\right)=(-\Delta +V)\left( \int_0^{t_k}(1-\frac{t}{t_k})g_k(t)dt\right).
$$
In \eqref{id:decompositionutY}, we introduce the notation $c_k=-t_k^{-1}\alpha_k(t_k)$ for all $k\geq k_0$, and $\tilde c_k=-t_k^{-1} \tilde \alpha_k(t_k)$ for\footnote{If $k_0\geq 0$ then $\tilde c_k=0$ for all $k$ by convention.} $k_0\leq k\leq -1$, and $\tilde c_k=0$ for $k\geq 0$, and we obtain the following identity for all $r\geq R_k$ if $k\geq 0$, and for all $R_k\leq r\leq 2$ if $k\leq -1$:
$$
(-\Delta+V) u_0= c_k \Lambda W+\tilde c_{k} \Gamma -\frac{1}{t_k}g_k(t_k)-(-\Delta +V)\left( \int_0^{t_k}(1-\frac{t}{t_k})g_k(t)dt\right).
$$
We equal the two above identities obtained for $k$ and $k+1$ respectively on the interval $[2R_k,\infty)$ if $k\geq 0$, and on the interval $[2R_k,2)$ if $k\leq -1 $, giving:
\begin{align}
\label{id:computationconstants}&(c_{k+1}-c_k) \Lambda W+(\tilde c_{k+1}-\tilde c_k)\Gamma \ = \ -\frac{1}{t_k}g_k(t_k)+\frac{1}{t_{k+1}} g_{k+1}(t_{k+1})\\
 \nonumber  &\qquad \qquad \qquad \qquad + (-\Delta +V)\left( \int_0^{t_{k+1}}(1-\frac{t}{t_{k+1}})g_{k+1}(t)dt- \int_0^{t_{k}}(1-\frac{t}{t_{k}})g_{k}(t)dt \right).
\end{align}
If $k\geq -1$, we integrate \eqref{id:computationconstants} against $\Phi_k$, and find after using $\tilde c_k=0$ for $k\geq 0$, integrating by parts and estimating using \eqref{bd:ginterufromut}, that $(c_{k+1}-c_k)R_k^2=O(R_k^2\| \partial_tu_L\|_{Y_R})$, and hence:
$$
|c_{k+1}-c_k|\lesssim \| \pa_tu_L\|_{Y_R} \qquad \mbox{for all }k\geq -1.
$$
If $k\leq -2$, we integrate \eqref{id:computationconstants} against $\Phi_{-1}$ and find after integrating by parts and estimating using \eqref{bd:ginterufromut}, that $(c_{k+1}-c_k)\frac{1}{4}=O(R_k^{-1}\|  \pa_tu_L\|_{Y_R})$, so that:
$$
|c_{k+1}-c_k|\lesssim R_k^{-1}\|  \pa_tu_L\|_{Y_R} \qquad \mbox{for all }k\leq -2.
$$
Let $k_1=\max (k_0,0)$. We deduce from the two inequalities above that for all $k\in \mathbb Z$:
\begin{equation}\label{bd:ckinterufromut}
c_k=c_{k_1}+ c_k', \qquad \begin{array}{l l}  \mbox{with, if } k_0 \leq 0, \qquad c_k'=\left\{ \begin{array}{l l} O(k\|  \pa_tu_L\|_{Y_R}) \qquad \mbox{for }k\geq 0, \\ O(R_k^{-1}\|  \pa_tu_L\|_{Y_R})\qquad \mbox{for }k_0\leq k\leq -1,  \end{array} \right. \\ \\ \mbox{or, if  } k_0 > 0 \qquad c_k'= O(|k-k_0| \|  \pa_tu_L\|_{Y_R})) \quad \mbox{for } \quad k\geq k_0. \end{array}
\end{equation}
If $k\leq -1$, we integrate \eqref{id:computationconstants} against $\tilde \Phi_{k}$ and find, after integrating by parts and estimating using \eqref{bd:ginterufromut}, that $(\tilde c_{k+1}-\tilde c_k)R_k^2=O(R_k^{2}\|  \pa_tu_L\|_{Y_R})$, and get that:
$$
|\tilde c_{k+1}-\tilde c_k|\lesssim \|  \pa_tu_L\|_{Y_R}.
$$
We deduce from the inequality above and from the fact that $\tilde{c}_0=0$, that, when $k_0\leq -1$, then for all $k_0\leq k\leq -1$:
\begin{equation}\label{bd:ckinterufromut2}
|\tilde c_k|\lesssim |k|  \|  \pa_tu_L\|_{Y_R}.
\end{equation}
We then get the following identity for all $k\geq k_0$ and $R_k\leq r \leq 3R_k$:
$$
-\Delta u_0+V u_0= c_{k_1} \Lambda W+\underbrace{c_k' \Lambda W+\tilde c_{k}\Gamma -\frac{1}{t_k}g_k(t_k)}_{=f^1_k}-(-\Delta +V)\left( \int_0^{t_k}(1-\frac{t}{t_k})g_k(t)dt\right).
$$
We compute the commutator relation for $k\in \mathbb Z$:
\begin{align*}
&(-\Delta +V) \left(\chi_k \left( \int_0^{t_k}(1-\frac{t}{t_k})g_k(t)dt\right)\right) -\chi_k (-\Delta +V) \left( \int_0^{t_k}(1-\frac{t}{t_k})g_k(t)dt\right)\\
&= \underbrace{R_k^{-2}(2\chi''-\Delta \chi) (\frac{r}{R_k})  \int_0^{t_k}(1-\frac{t}{t_k})g_k(t)dt }_{=f^2_k}+\pa_r \underbrace{\left( -2R_k^{-1}\chi' (\frac{r}{R_k} )   \int_0^{t_k}(1-\frac{t}{t_k})g_k(t)dt\right)}_{= \tilde g_k}.
\end{align*}
Using the $\chi$ based partition of unity, we have that on $[\frac 32 R_{k_0},\infty)$:
\begin{align*}
& (-\Delta +V) u_0 \ = \ \sum_{k\geq k_0}\chi_k(-\Delta +V)u_0 \\
&=c_{k_1}\Lambda W+ \sum_{k\geq k_0} (\chi_k f_k^1+f_k^2)+\pa_r \left(\sum_{k\geq k_0} \tilde g_k\right)-(-\Delta +V)\left(\sum_{k\geq k_0} \chi_k \left( \int_0^{t_k}(1-\frac{t}{t_k})g_k(t)dt\right)\right).
\end{align*}
We decompose:
\begin{equation}\label{eq:decompositionu0wave1}
u_0=c_{k_1}\Upsilon+\bar u_0+\tilde u_0, \qquad \bar u_0=-\sum_{k\geq k_0} \chi_k \left( \int_0^{t_k}(1-\frac{t}{t_k})g_k(t) dt\right).
\end{equation}
The new unknown $\tilde u_0$ solves for $r\geq \frac 32 R_{k_0}$:
\begin{equation}\label{eq:elliptictildeu01}
(-\Delta +V )\tilde u_0 = f+\pa_r g, \ \ f=\sum_{k\geq k_0}\left(\chi_k f_k^1+f_k^2\right), \ \ g=\sum_{k\geq k_0} \tilde g_k .
\end{equation}

\noindent \textbf{Step 2.} \emph{Solving the elliptic equation}. We estimate each term in \eqref{eq:elliptictildeu01}. For the first one, using \eqref{bd:ckinterufromut}, \eqref{bd:ckinterufromut2} and \eqref{bd:ginterufromut}, and the definition of $k_0$ and $k_1$, for all $k\geq k_0$:
$$
\| f_k^1\|_{L^2(R_k\leq r\leq 3R_k)}\lesssim   R_k^{-1} \langle \log \frac{R_k}{\langle R\rangle} \rangle \| \pa_t u\|_{Y_R}.
$$
For the second one, using \eqref{bd:ginterufromut} we estimate:
$$
\| f_k^2\|_{L^2(R_k\leq r\leq 3R_k)}\lesssim R_k^{-2} \int_0^{t_k} \| g_k (t)\|_{L^2(R_k\leq r\leq 3R_k)}dt\lesssim R_k^{-1} \| \pa_t u\|_{Y_R}.
$$
Hence
\begin{equation}\label{bd:interwavef1}
\| f\|_{Z_{-4,R_{k_0}}}\lesssim   \| \pa_t u\|_{Y_R} .
\end{equation}
We next estimate $\tilde g_k$ using \eqref{bd:ginterufromut}:
$$
\| \tilde g_k\|_{L^2(R_k\leq r\leq 3R_k)}\lesssim R_k^{-1} \int_0^{t_k} \| g_k (t)\|_{L^2(R_k\leq r\leq 3R_k)}dt\lesssim  \| \pa_t u\|_{Y_R}.
$$
so that:
\begin{equation}\label{bd:interwaveg1}
\| g\|_{\widetilde{Z}_{-3,R_{k_0}}}\lesssim   \| \pa_t u\|_{Y_R}.
\end{equation}
We estimate similarly to \eqref{bd:interwaveg1} that:
\begin{equation}\label{bd:interwavebaru01}
\| \bar u_0 \|_{\widetilde{Z}_{-2,R_{k_0}}}\lesssim   \| \pa_t u\|_{Y_R}.
\end{equation}
We now consider the equation \eqref{eq:elliptictildeu01} for $\tilde u_0$, and apply Lemma \ref{L:ellipticunsoliton}, using the estimate \eqref{bd:interwavef1} and \eqref{bd:interwaveg1} and get:
$$
d_{_{Z_{-2,\frac 32R_{k_0}}}} ( \tilde u_0, \textup{Span}(\Lambda W,\Upsilon)) \lesssim    \| \pa_t u\|_{Y_R}.
$$
Injecting the above inequality, using $\frac 32 R_{k_0}\leq 6R$, in the decomposition \eqref{eq:decompositionu0wave1}, using \eqref{bd:interwavebaru01}, gives the desired bound of Lemma \ref{L:estimateufromutunsoliton}.

\end{proof}

The next Lemma upgrades the weighted $L^2$ bound of Lemma \ref{L:estimateufromutunsoliton} into a weighted $L^2$ bound for the gradient.

\begin{lemma}[Estimating $\nabla u$ from $u_0$ and $\pa_t u$ for waves around a soliton] \label{lem:estimateufromutunsoliton2}

There exists $C>0$ such that the following holds true. Assume that $u_L$ solves \eqref{LW}, with $\vec u(0)\in \mathcal H$. Then for any $R>0$:
$$
d_{Z_{-3,24R}}(\pa_r u_0,\text{Span}(\pa_r \Lambda W,\pa_r \Upsilon))\lesssim d_{Z_{-2,6R}}(u_0,\text{Span}(\Lambda W,\Upsilon)) +  \| \partial_t u\|_{Y_R}.
$$

\end{lemma}

\begin{proof}

We let $q=d_{Z_{-2,6R}}(u_0,\text{Span}(\Lambda W,\Upsilon)) +  \| \partial_tu\|_{Y_R}$ to ease notations. There exist $a_0,\bar a_0\in \mathbb R$ such that $\| u_0-a_0\Lambda W-\bar a_0 \Upsilon\|_{Z_{-2,6R}}\leq 2q$. Consider $u'=u+(\bar a_0\frac{t^2}{2}-a_0)\Lambda W-\bar a_0\Upsilon$. Then $u'$ also solves \eqref{LW}, with $\| u'_0\|_{Z_{-2,6R}}+ \| u'_t\|_{Y_R}\leq 3q$. Hence, it suffices to prove\footnote{We do a slight abuse since $u_0'$ might not belong to $\mathcal H$, but all the computations of the proof are nonetheless valid.} the result of the lemma for $u_0'$. Without loss of generality, we thus assume that $a_0=\bar a_0=0$ so that:
\be \label{id:decompositionu0inter}
\| u_0\|_{Z_{-2,6R}}\lesssim q.
\ee 

We let for $k\in \mathbb Z$, $R_k=2^k$ and $t_k=R_k/2$. Let $k_0$ be such that $R_{k_0-1}<3R\leq R_{k_0} $. We recall the decomposition \eqref{id:decompositionutY}.\\

\noindent \textbf{Step 1}. \emph{An energy identity for a projection of $u$}. We have the following energy type equality for suitable radial functions, which is obtained by performing integration by parts:
\be \label{id:L2toH1logenergyidentity}
\iint \psi  (\pa_{tt}v -\Delta v +V v)v  dt dx = \iint (|\nabla v|^2-(\pa_t v)^2)\psi dtdx +\iint \Big(V+\frac 12 (\pa_{tt} -\Delta )\Big)\psi v^2 dtdx.
\ee
We let $\chi$ be a smooth nonnegative radial cut-off function, with $\chi(r)=1$ for $r\in [3, 9]$ and $\chi(r)=0$ for $r\leq 2$ and $r\geq 10$. We let $\tilde \chi$ be a smooth one-dimensional cut-off, $\tilde \chi(t)=1$ for $|t|\leq 1/2$ and $\tilde \chi(t)=0$ for $|t|\geq 1$. We define $\chi_k (r)=\chi(r/R_k)$ and $\tilde \chi_k(t)=\tilde \chi(t/R_k)$, and $\psi_k(t,r)=\tilde \chi_k(t)\chi_k(r)$.

We now pick $k\geq k_0$. Observe that there holds $\textup{supp}(\chi_{R_k})\subset \{r\geq 6R\}$ and $\textup{supp}(\psi_k) \subset \{r\geq |t|+R\}$. We introduce 
$$\langle u,v\rangle_{L^2_{\chi_k}}= \int uv\chi_k,\quad \| u\|_{L^2_{\chi_k}}=\sqrt{\langle u,u\rangle_{\chi_k}}.$$ 
For each $k\in \mathbb Z$, and $|t|\leq R_k$, we let $b_k(t)$ and $\tilde b_k(t)$, with $\tilde b_k=0$ for $k\geq 0$ by convention, be the unique parameters such that:
\be \label{id:decompositionuinter}
u_L(t) =v_k(t)+b_k(t)\Lambda W+\tilde b_k(t)\tilde \Gamma,
\ee
where we recall that $\tilde \Gamma$ is defined by \eqref{id:deftildeGamma}, and $v_k$ satisfies the orthogonality conditions:
\be \label{bd:L2toH1loginterortho}
 \int v_k(t) \Lambda W \chi_k dx=0 \quad \mbox{for all }k\in \mathbb Z, \quad \mbox{ and } \int v_k(t) \tilde \Gamma \chi_k dx=0 \quad \mbox{for all }k\leq -1.
\ee
Then, using $(-\Delta +V)\Lambda W=0$, and $(-\Delta +V)\tilde \Gamma=0$ for $r\leq 10$ from \eqref{id:deftildeGamma}, we obtain that $v_k$ solves the following equation on the support of $\chi_k$, that is, on $\{ 2R_k \leq r\leq 10R_k \}$:
$$
\pa_{tt} v_k=\Delta v_k -V v_k -\pa_{tt}b_k \Lambda W-\pa_{tt}\tilde b_k \tilde \Gamma.
$$
Hence, from \eqref{id:L2toH1logenergyidentity} and the orthogonality conditions \eqref{bd:L2toH1loginterortho}, we get the energy identity for $v_k$:
\be \label{id:L2toH1logenergyidentityv}
\iint |\nabla v_k|^2 \psi_k dtdx =\iint (\pa_t v_k)^2\psi_k dtdx -\iint \Big(V+\frac 12 (\pa_{tt} -\Delta )\Big)\psi_kv_k^2 dtdx.
\ee

\noindent \textbf{Step 2}. \emph{Estimates for $v_k$ and $\pa_t v_k$}. We claim that, for all $|t|\leq R_k$:
\be \label{id:L2toH1bornesintervk}
 \| v_k(t) \|_{L^2(2R_k\leq r \leq 10R_k)} \lesssim R_k  \langle \log \frac{R_k}{\langle R\rangle} \rangle q   \quad \mbox{and}\quad  \|  \pa_t v_k(t) \|_{L^2(2R_k\leq r \leq 10R_k)} \lesssim q .
\ee
To show it, using the decompositions \eqref{id:decompositionutY} and \eqref{id:decompositionuinter}, we write:
\begin{align} \label{id:L2toH1bornesintervk0decomposition}
& u_0=v_k(0)+\left[b_k(0)\Lambda W+\tilde b_k(0)\tilde \Gamma \right],\\
\label{id:L2toH1bornesintervk0decomposition2}
& g_{k}(t)=\pa_t v_k(t)+\left[(\pa_t b_k(t)-\alpha_{R_k}(t))\Lambda W+(\pa_t \tilde b_k(t)-\tilde \alpha_{R_k}(t))\tilde \Gamma\right]
\end{align}
where the second equality is for $|t|\leq R_k$. Above, observe that:
\begin{align} \label{id:L2toH1bornesintervk0tech}
&  \| u_0\|_{L^2_{\chi_k}}\lesssim  \| u_0\|_{L^2(2R_k\leq r \leq 10R_k)} \lesssim R_k \langle \log \frac{R_k}{\langle R\rangle}\rangle \| u_0 \|_{Z_{-2,R}}\leq  R_k  \langle \log \frac{R_k}{\langle R\rangle} \rangle q,\\
 \label{id:L2toH1bornesintervk0tech2}
 & \| g_{k}\|_{L^2_{\chi_k}}\lesssim  \| g_k\|_{L^2(2R_k\leq r \leq 10R_k)} \lesssim  \| g_k\|_{L^2_{2R_k}}  \leq  \| \partial_tu\|_{Y_R} \leq   q.
\end{align}
The two terms in the right-hand sides of \eqref{id:L2toH1bornesintervk0decomposition}, and of \eqref{id:L2toH1bornesintervk0decomposition2} respectively, are orthogonal for the bilinear form $\langle\cdot,\cdot\rangle_{L^2_{\chi_k}}$ from \eqref{bd:L2toH1loginterortho}. Thus, by Pythagoras, \eqref{id:L2toH1bornesintervk0tech} and \eqref{id:L2toH1bornesintervk0tech2}:
\begin{align*}
&\| b_k(0)\Lambda W+\tilde b_k(0)\tilde \Gamma   \|_{L^2_{\chi_k}}\leq \| u_0\|_{L^2_{\chi_k}} \lesssim  R_k  \langle \log \frac{R_k}{\langle R \rangle} \rangle q,\\
& \|  (\pa_t b_k(t)-\alpha_{R_k}(t))\Lambda W+(\pa_t \tilde b_k(t)-\tilde \alpha_{R_k}(t))\Lambda W  \|_{L^2_{\chi_k}} \leq \| g_{k}(t)\|_{L^2_{\chi_k}} \lesssim   q.
\end{align*}
Given the behaviours of $\Lambda W$ and $\tilde \Gamma$, notice that for all general $c_1,c_2\in \mathbb R$:
$$
\| c_1 \Lambda W+c_2 \tilde \Gamma   \|_{L^2(2R_k\leq r\leq 10R_k)} \lesssim \| c_1\Lambda W+c_2 \tilde \Gamma   \|_{L^2_{\chi_k}}.
$$
Combining, we get:
\begin{align}\label{id:L2toH1bornesintervk0tech3}
&\| b_k(0)\Lambda W+\tilde b_k(0)\tilde \Gamma   \|_{L^2(2R_k\leq r\leq 10R_k)} \lesssim  R_k  \langle \log \frac{R_k}{\langle R\rangle} \rangle q,\\
\label{id:L2toH1bornesintervk0tech4} & \|  (\pa_t b_k(t)-\alpha_{R_k}(t))\Lambda W+(\pa_t \tilde b_k(t)-\tilde \alpha_{R_k}(t))\Lambda W  \|_{L^2(2R_k\leq r\leq 10R_k)} \lesssim   q.
\end{align}
Injecting \eqref{id:L2toH1bornesintervk0tech} and \eqref{id:L2toH1bornesintervk0tech3} in \eqref{id:L2toH1bornesintervk0decomposition}, and injecting \eqref{id:L2toH1bornesintervk0tech2} and \eqref{id:L2toH1bornesintervk0tech4} in \eqref{id:L2toH1bornesintervk0decomposition2}, one obtains the first inequality in \eqref{id:L2toH1bornesintervk} at $t=0$, and the second inequality in \eqref{id:L2toH1bornesintervk}. Combined together, they imply the first inequality in \eqref{id:L2toH1bornesintervk} for all $|t|\leq R_k$. Hence \eqref{id:L2toH1bornesintervk} is established.\\

\noindent \textbf{Step 3}. \emph{Estimate for $\nabla v_k$}. Using the bounds \eqref{id:L2toH1bornesintervk}, the energy identity \eqref{id:L2toH1logenergyidentityv} gives:
$$
\iint |\nabla v_k|^2 \psi_k dtdx \lesssim R_k \langle \log \frac{R_k}{\langle R \rangle} \rangle^2 q^2.
$$
By the mean value Theorem and the definition of $\psi_k$, there exists a time $|t_k|\leq R_k/2$ such that:
\be \label{id:L2toH1bornesintervk2}
\int_{3R_k\leq |x|\leq 9R_k} |\nabla v(t_k)|^2 dx \lesssim  \langle \log \frac{R_k}{\langle R \rangle} \rangle^2 q^2 .
\ee

\noindent \textbf{Step 4}. \emph{Estimate for $\nabla u_0$}. For all $k\in \mathbb Z$, we have at the time $t_k$ of Step 3, using \eqref{id:decompositionutY}, \eqref{id:decompositionuinter}:
\begin{align*}
& u_L(t_k)=v_k(t_k)+b_k(t_k)\Lambda W+\tilde b_k(t_k)\tilde \Gamma, \\
& \pa_t u_L(t_k)=g_{k}(t_k)+\alpha_{k}(t_k)\Lambda W+\tilde \alpha_k(t_k)\tilde \Gamma.
\end{align*}
where we use the convention that $\tilde \alpha_k(t_k) =0$ for $k\geq 0$. Hence, by finite speed of propagation, using \eqref{id:deftildeGamma}, for $|t|\leq t_k$ and $4R_k\leq r\leq 8R_k$, we have 
$$
u_L(t)=w_k(t)+(b_k(t_k)+(t-t_k)\alpha_{k}(t_k))\Lambda W+(\tilde b_k(t_k)+(t-t_k)\tilde \alpha_{k}(t_k))\Gamma
$$
where $w_k$ is the solution to
$$
\pa_{tt}w_k=\Delta w_k-Vw_k, \quad w_k(t_k)=v_k(t_k), \quad \pa_t w_k(t_k)=g_{k}(t_k).
$$
Introducing $d_k=b_k(t_k)-t_k\alpha_{R_k}(t_k)$ and $e_k=\tilde b_k(t_k)-t_k\tilde \alpha_{R_k}(t_k)$ we obtain that for $4R_k\leq r\leq 8R_k$:
\be \label{id:L2toH1idinteru}
u_0=w_k(0)+d_k \Lambda W+e_k\Gamma
\ee
Let an extension $(\tilde w_{k,0},\tilde w_{k,1})$ be such that $(\tilde w_{k,0},\tilde w_{k,1})(r)=(v_k,g_k)(t_k,r)$ for all $r\in [3R_k,9R_k]$ with
\be \label{L2toH1:bd:estimationtildewk}
\left\|\left( \tilde w_{k,0},\tilde w_{k,1}\right)\right\|_{\mathcal H}^2\lesssim \int_{3R_k\leq |x|\leq 9R_k} (|\nabla v_k(t_k)|^2+R_k^{-2}|v_k(t_k)|^2+|g_k(t_k)|^2)dx,
\ee
and let $\tilde w_k$ solve $\pa_t^2 \tilde w_k-\Delta \tilde w_k+V\tilde w_k=0$ with data $(\tilde w_k(t_k),\pa_t \tilde w_k(t_k))=(\tilde w_{k,0},\tilde w_{k,1})$.
Then by finite speed of propagation, as $|t_k|\leq R_k/2$ we obtain $\tilde w_k(0)=w_k(0)$ for all $4R_k\leq r \leq 8R_k$. Applying standard energy estimates, and then using \eqref{L2toH1:bd:estimationtildewk} with \eqref{id:L2toH1bornesintervk2}, \eqref{id:L2toH1bornesintervk} and \eqref{bd:ginterufromut}:
$$
\| \tilde w_k(0)\|_{\dot H^1_{R_k}}\lesssim \| ( \tilde w_{k,0},\tilde w_{k,1}) \|_{\mathcal H}\lesssim \langle \log \frac{R_k}{\langle R \rangle} \rangle q.
$$
Hence, using Hardy, and then that $\tilde w_k(0)=w_k(0)$ for all $4R_k\leq r \leq 8R_k$:
\be \label{id:L2toH1bornesinterwk}
\| \nabla w_k(0)\|_{L^2(4R_k\leq r\leq 8R_k))}+R_k^{-1}\| w_k(0)\|_{L^2(4R_k\leq r\leq 8R_k))} \lesssim  \langle \log \frac{R_k}{\langle R\rangle} \rangle q.
\ee
Using the definition of the $Z_{-2}$ norm, \eqref{id:L2toH1bornesinterwk} and \eqref{id:L2toH1idinteru} we obtain:
$$
\| d_k \Lambda W+e_k \Gamma\|_{L^2(4R_k\leq r \leq 8R_k)}\lesssim R_k \langle \log \frac{R_k}{\langle R\rangle} \rangle q ,
$$
and so using \eqref{id:asymptoticGamma}: $\| d_k \Lambda W\|_{L^2(4R_k\leq r \leq 8R_k)}+\| e_k \Gamma  \|_{L^2(4R_k\leq r \leq 8R_k)} \lesssim R_k \langle \log \frac{R_k}{\langle R\rangle} \rangle q , $
so that, using \eqref{id:asymptoticparGamma}:
$$
\| \nabla d_k \Lambda W\|_{L^2(4R_k\leq r \leq 8R_k)}+\|\nabla  e_k \Gamma  \|_{L^2(4R_k\leq r \leq 8R_k)} \lesssim  \langle \log \frac{R_k}{\langle R\rangle} \rangle q.
$$
Injecting the above inequality and \eqref{id:L2toH1bornesinterwk} in \eqref{id:L2toH1idinteru} one obtains:
$$
\| \nabla u_0 \|_{L^2(4R_k\leq r \leq 8R_k)}\lesssim \langle \log \frac{R_k}{\langle R\rangle} \rangle q.
$$
This implies the desired result of the Lemma, as $4R_{k_0}\leq 24R$.

\end{proof}

The proofs of the main results of this subsection are complete. Indeed,
Lemma \ref{L:channelsevenunsoliton2} follows from Lemmas \ref{L:channelsoddunsoliton}, 
 \ref{L:estimateufromutunsoliton} and \ref{lem:estimateufromutunsoliton2}.
Theorem \ref{th:channelsunsoliton} is a direct consequence of Lemmas \ref{L:channelsoddunsoliton} and \ref{L:channelsevenunsoliton}.

\section{Channels of energy close to a multisoliton} \label{subsec:channels_multi}

First, note that the solution $u$ of \eqref{id:linearmultisoliton} is globally well-posed in $\HHH$: the local well-posedness can be proved by Strichartz estimates \eqref{CK1} and the fact that $W$ is in $L^4$, and the global well-posedness follows from linearity. The exterior energy \eqref{id:exteriorenergy} is well-defined, as an application of Lemma \ref{lem:linearmultiout} using finite speed of propagation.

We divide the proof of Proposition \ref{pr:channels} into two parts. In \S \ref{Sub:odd} (see Lemma \ref{lem:channelsodd}), we prove the result for odd in time solutions. In \S \ref{Sub:even} we consider even in time solutions (see Lemma \ref{L:channelseven}). 
 
We start by some technical preliminaries on the equation   \eqref{id:linearmultisoliton} and on the spaces $Z_{\alpha,\lambdabf}$ that will be needed in the proof.
\subsection{Preliminaries}
\begin{lemma} \label{lem:linearmultiout}
For all $J\in \mathbb N$, a constant $C>0$ exists such that for all $\lambdabf \in \Lambda_J$ and $R\geq 0$, if $u$ solves on $\mathbb R^{1+6}$:
$$
\Box u+V_{\lambdabf}u=f, \qquad \vec u(0)\in \mathcal H
$$
where $f\in L^1_tL^2_x(\mathbb R\times \mathbb R^6)$ then:
$$
\sup_{t\in \mathbb R} \| \vec u(t)\|_{\mathcal H_{R+|t|}}+\| u\|_{L^2_tL^4(r\geq R+|t|)}\leq C\left(\| \vec u(0)\|_{\mathcal H_R}+\| f\|_{L^1_tL^2_x(r\geq R+|t|)} \right).
$$

\end{lemma}

\begin{proof}

The proof is the same as that of Lemma 2.8 in \cite{DuKeMe19Pb}, we omit it.

\end{proof}

\begin{claim}
\label{claimZ}
 Let $f\in\dot{H}^1$, radial. Then $|\nabla f|\in  Z_{-3,\lambdabf}$ with
 \begin{equation}
\label{embeddingH1Z-3}
 \| \nabla f \|_{Z_{-3,\lambdabf}}\leq  \| f\|_{\dot H^1},
 \end{equation}
and
 \begin{equation}
\label{Sobolev_Z}
|f(r)|\lesssim \frac{1}{r^2}\left( 1+\left|\min_{1\leq j\leq J}\log\left(\frac{r}{\lambda_j}\right)\right| \right)\|\nabla f\|_{Z_{-3,\lambdabf}}.
 \end{equation} 
 Furthermore, assume that $\chi^0$ and $\chi^1$ are two cut-off functions, with $\chi^0(r)=1$ for $r\leq 1$ and $\chi^0(r)=0$ for $r\geq 2$, and $\chi^1$ that is compactly supported outside the origin, and denote $\chi^i_R(r)=\chi^i(r/R)$ for $R>0$ and $i=0,1$, then we have:
 \begin{equation}
\label{localisingZ-3}
 \| \chi^1_Rf \|_{\dot H^1}\leq C  \left( 1+\left|\min_{1\leq j\leq J}\log\left(\frac{R}{\lambda_j}\right)\right|\right) \| \nabla f\|_{Z_{-3,\lambdabf}}
 \end{equation}
 and
  \begin{equation}
\label{localisingZ-32}
 \| \nabla (\chi^0_R f) \|_{Z_{-3,\lambdabf}}\leq C \| \nabla f\|_{Z_{-3,\lambdabf}}, \qquad  \| \nabla ((1-\chi^0_R) f) \|_{Z_{-3,\lambdabf}}\leq C  \| \indic_{\{|x|\geq R\}}\nabla f\|_{Z_{-3,\lambdabf}}
 \end{equation}
 for $C$ depending on $\chi^0$ and $\chi^1$ but independent of $R$.
\end{claim}

\begin{proof}
The inequality \eqref{embeddingH1Z-3} is a direct consequence of the inequality $\| \nabla f\|_{Z_{-3,\lambdabf}}\leq \| \nabla f\|_{L^2}$.

Next, we have 
 \begin{multline*}
  |f(r)|\lesssim \sum_{k\geq 0} \int_{2^kr}^{2^{k+1}r} \left|\frac{\partial f}{\partial\rho} \right|d\rho\lesssim \sum_{k\geq 0} \frac{1}{2^{2k}r^2}\left(\int_{2^kr}^{2^{k+1}r} \left|\frac{\partial f}{\partial\rho}\right|^2\rho^5d\rho\right)^{1/2}\\
  \lesssim \sum_{k\geq 0}\frac{1}{(2^kr)^2} \|\nabla f\|_{Z_{-3,\lambdabf}}\left( 1+\min_{j\in \llbracket 1,J\rrbracket} |\log (2^kr/\lambda_j)|\right)\\
  \lesssim \sum_{k\geq 0}\frac{1}{(2^kr)^2} \|\nabla f\|_{Z_{-3,\lambdabf}}\left( k+\min_{j\in \llbracket 1,J\rrbracket} |\log (r/\lambda_j)|\right),
 \end{multline*}
which yields \eqref{Sobolev_Z}. Then, using Leibniz, \eqref{Sobolev_Z}, and the compact support of $\chi^1$ outside the origin, we infer:
 \begin{multline*}
\| \chi_R^1 f\|_{\dot H^1}\lesssim \| (\nabla \chi_R^1) f\|_{L^2}+\| \chi_R^1 \nabla f \|_{L^2}\\
\lesssim \left\| r^{-2}\nabla \chi_R^1\right\|_{L^2} \left( 1+\left|\min_{1\leq j\leq J}\log\left(\frac{R}{\lambda_j}\right)\right|\right) \| \nabla f\|_{Z_{-3,\lambdabf}}+ \left( 1+\left|\min_{1\leq j\leq J}\log\left(\frac{R}{\lambda_j}\right)\right|\right) \| \nabla f\|_{Z_{-3,\lambdabf}}\\
\lesssim \left( 1+\left|\min_{1\leq j\leq J}\log\left(\frac{R}{\lambda_j}\right)\right|\right) \| \nabla f\|_{Z_{-3,\lambdabf}},
 \end{multline*}
which shows \eqref{localisingZ-3}. The proof of \eqref{localisingZ-32} is similar and we omit it.

\end{proof}

\begin{lemma}

There exists $C>0$ such that, for any $\lambda>0$, for any $u_0\in \dot H^1$, the solution $u$ to $\pa_t^2 u-\Delta u+V_{(\lambda)}u$ with initial data $(u(0),\pa_t u(0))=(u_0,0)$ satisfies:
\begin{equation} \label{interactionsolitonZ-31}
\| Vu\|_{L^1L^2(r>|t|)}\leq C \| \nabla u_0\|_{Z_{-3}}.
\end{equation}
If moreover $u_0(r)=0$ for all $r\geq R$, for some $R\leq 1$, then:
\begin{equation} \label{interactionsolitonZ-32}
\| Vu\|_{L^1L^2(r>|t|)}\leq CR^{\frac 18} \| \nabla u_0\|_{Z_{-3}}.
\end{equation}
If moreover, $u_0(r)=0$ for all $r\leq R$, for some $R\geq 1$, then:
\begin{equation} \label{interactionsolitonZ-33}
\| Vu\|_{L^1L^2(r>|t|)}\leq CR^{-1} \| \nabla u_0\|_{Z_{-3}}.
\end{equation}

\end{lemma}

\begin{proof}

We use a dyadic partition of unity $1=\sum_{\ell\in \mathbb Z} \chi_{2^{\ell}}(r)$, with $\textup{supp}(\chi)\subset [\frac 12,2]$. We decompose $ u=\sum_{\ell\in \mathbb Z} u_{\ell}$, where $u_{\ell}$ is the solution to $\pa_t^2 u_{\ell}-\Delta u_{\ell}+V_{(\lambda)}u_{\ell}=0$ with initial data $(u_{\ell}(0),\pa_t u_{\ell}(0))=(u_{\ell,0},0)$, where $u_{\ell,0}=\chi_{2^{\ell}}u_0$. By \eqref{localisingZ-3} (with $J=1$, $\lambda_1=1$):
$$
\| u_{\ell,0}\|_{\dot H^1}\lesssim (1+|\ell|) \| \nabla u_0\|_{Z_{-3}}.
$$
Since $\textup{supp}(u_{\ell,0})\subset [2^{\ell-1},2^{\ell+1}]$, we have by finite speed of propagation that $\textup{supp}(u_{\ell,0})\subset \Big\{\max(0,2^{\ell-1}-|t|\leq r \leq 2^{\ell+1}+|t|\Big\}$. Using H\"older and Lemma \ref{lem:linearmultiout}, we deduce:
$$
\| V u_{\ell}\|_{L^1L^2(r\geq |t|)}\lesssim  (1+|\ell|)\| V\|_{L^2_tL^4_x(S_{\ell})} \| \nabla u_0\|_{Z_{-3}}
$$
where we introduced the set $S_{\ell}=\Big\{\max(|t|,2^{\ell-1}-|t|)\leq r \leq 2^{\ell+1}+|t|\Big\}$. If $\ell\leq 0$ then:
$$
\| V\|_{L^2_tL^4_x(S_{\ell})}\leq \| V\|_{L^2_tL^4_x(|t|\leq r \leq 2^{\ell+1}+ |t|)}\lesssim 2^{\frac \ell4}
$$
by \eqref{estim5}, while if $\ell\geq 1$ then, since $\max(|t|,2^{\ell-1}-|t|)\geq \max (|t|,2^{\ell-2})$:
$$
\| V\|_{L^2_tL^4_x(S_{\ell})}\leq \| V\|_{L^2_tL^4_x(r\geq \max(|t|,2^{l-2}))}\lesssim 2^{-2\ell}
$$
by \eqref{estim6}. Summing, we get
\begin{multline*}
\| V u\|_{L^1L^2(r\geq |t|)}\lesssim \sum_{\ell\in Z} \| V u_{\ell}\|_{L^1L^2(r\geq |t|)} \\
\lesssim   \| \nabla u_0\|_{Z_{-3}}\left(\sum_{\ell\leq 0} (1+|\ell|)2^{\frac l4}+\sum_{\ell\geq 1}(1+\ell)2^{-2\ell} \right)\lesssim \| \nabla u_0\|_{Z_{-3}}.
\end{multline*}
This is \eqref{interactionsolitonZ-31}. If in addition $u_0(r)=0$ for all $r\geq R$ for some $R\leq 1$, then, introducing $\ell_0$ the only integer such that $2^{\ell_0-1}< R\leq 2^{\ell_0}$, we have $u_{\ell}=0$ for $\ell\geq \ell_0+1$ and so:
$$
\| V u\|_{L^1L^2(r\geq |t|)}\lesssim   \| \nabla u_0\|_{Z_{-3}}\sum_{\ell\leq \ell_0} (1+|\ell|)2^{\frac \ell 4}\lesssim 2^{\frac{\ell_0}{8}}\| \nabla u_0\|_{Z_{-3}}\lesssim R^{\frac{1}{8}}\| \nabla u_0\|_{Z_{-3}},
$$
which establishes \eqref{interactionsolitonZ-32}. The proof of \eqref{interactionsolitonZ-33} is similar and we omit it.

\end{proof}

\begin{lemma} \label{lem:approxresonances}
For any $J\in \mathbb N$, there exists $\gamma^*>0$ such that the following holds for all $\lambdabf \in \Lambda_J$ with $\gamma(\lambdabf)\leq \gamma^*$. For $1\leq j \leq J$ denote by $\phi_j$ the solution to \eqref{id:linearmultisoliton} with data $(0,(\Lambda W)_{[\lambda_j]})$, and by $\psi_j$ the solution to \eqref{id:linearmultisoliton} with data $((\Lambda W)_{(\lambda_j)},0)$.

Then for any $1\leq j \leq J$, 
\be \label{bd:phij}
 \sup_{t\in \mathbb R}\| \vec{\phi}_j(t)-(t\Lambda W_{[\lambda_j]},\Lambda W_{[\lambda_j]})\|_{\mathcal H_{|t|}}\lesssim \gamma(\lambdabf),
 \ee
 and
 \be \label{bd:psij}
 \sup_{t\in \mathbb R}\| \vec{\psi}_j(t)-(\Lambda W_{(\lambda_j)},0)\|_{\mathcal H_{|t|}}\lesssim \gamma(\lambdabf)^2|\log(\gamma(\lambdabf))|.
 \ee
\end{lemma}
\begin{proof}
 By scaling invariance, we can assume $\lambda_j=1$ and then $\lambda_{j+1}<1<\lambda_{j-1}$. We introduce $\bar \phi_j= \phi_j-t\Lambda W$ that solves:
\be \label{approx:eq:barphii}
\Box \bar \phi_j+V_{\lambdabf}\bar \phi_j=-\sum_{i\neq j}t  \Lambda W V_{(\lambda_i)}, \qquad (\bar \phi_j(0),\pa_t \bar \phi_j(0))=\vec 0.
\ee
% We recall that $|V(r)|+|\Lambda W (r)|\lesssim \langle r\rangle^{-4}$ for $r>0$. If $i>j$ then for all $r\geq 0$ and $t\in \mathbb R$:
% $$
% |t\Lambda W(r)V_{(\lambda_i)}(r)|\lesssim \frac{t}{\lambda_i^2}\indic(r\leq \lambda_i)+\frac{t\lambda_i^2}{r^4}\indic (\lambda_i\leq r \leq 1)+\frac{t\lambda_i^2}{r^8} \indic (r \geq 1)
% $$
% so that by direct computations:
By Appendix \ref{A:estimates}, for $i>j$,
$\| t\Lambda WV_{(\lambda_i)} \|_{L^1_tL^2_x(\{r\geq |t|)\}}\lesssim \lambda_i^2$, and for $i<j$, $\| t\Lambda WV_{(\lambda_i)} \|_{L^1_tL^2_x(\{r\geq |t|)\}}\lesssim \frac{1}{\lambda_{i}}$,
and Lemma \ref{lem:linearmultiout} 
gives \eqref{bd:phij}. The same proof gives \eqref{bd:psij}, using that by Appendix \ref{A:estimates}, 
$\| \Lambda WV_{(\lambda_i)} \|_{L^1_tL^2_x(\{r\geq |t|)\}}\lesssim \lambda_i^2|\log \lambda_i|$ for $i>j$, while $\| \Lambda WV_{(\lambda_i)} \|_{L^1_tL^2_x(\{r\geq |t|)\}}\lesssim \lambda_i^{-2}|\log \lambda_i|$ if $i<j$.
\end{proof}

\subsection{Channels of energy around a multisoliton for odd in time solutions}
\label{Sub:odd}
In this subsection we prove the following result:
\begin{lemma}[Channels for odd in time solutions around a multisoliton] \label{lem:channelsodd}
For any $J\in \mathbb N$, there exist $\gamma^*,C>0$ such that for all $\lambdabf \in \Lambda_J$ with $\gamma(\lambdabf)\leq \gamma^*$, any radially symmetric solution $u$ of \eqref{id:linearmultisoliton} satisfies:
\begin{equation}
\label{bound_sol'}
\| \Pi_{L^2,\lambdabf}^\perp \pa_t u(0)\|_{L^2}^2\leq C \left(\sum_{\pm } \lim_{t\rightarrow \pm \infty}  \int_{r\geq |t|} |\nabla_{t,x}u|^2 \ +\ \gamma(\lambdabf)^2 \|  \pa_t u(0) \|_{L^2}^2\right).
\end{equation} 
\end{lemma}
For fixed $R\geq 0$ and $\mu>0$, we will denote by $\Pi_{\mu,R}^{\perp}$ the orthogonal projection, in $L^2_R$ on the orthogonal of $\{(\Lambda W)_{[\mu]}\}$. We have
\begin{claim}
 \label{Cl:rescaled_unsoliton}
For all $\eta\geq 0$ there exists $C_{\eta}>0$ such that for all $\mu>0$ and $u$ a solution of 
 \begin{equation}
  \label{Th10}
  \partial_t^2u-\Delta u+V_{(\mu)}u=f,\quad |x|>\eta \mu+|t|,
 \end{equation} 
with $\vec{u}(0)\in \HHH$ and $f\in L^1L^2$, there holds
\begin{equation*}
 \left\|\Pi^{\bot}_{\mu,\eta \mu}\partial_tu(0)\right\|_{L^2_{\eta \mu}}^2\leq 
 C_\eta\left\| \indic_{\{|x|>\eta \mu+|t|\}} f\right\|^2_{L^1L^2}+C_{\eta}\sum_{\pm}\lim_{t\to\pm \infty}\int_{|x|>\eta\mu+|t|}|\nabla_{t,x}u(t,x)|^2dx.
\end{equation*} 
 \end{claim}
 \begin{proof}
  By scaling, energy estimates, time symmetry and finite speed of propagation, it is sufficient to prove the result assuming that $u(0)\equiv 0$, $\mu=1$ and $f\equiv 0$. If $\eta=0$, this is Lemma \ref{L:channelsoddunsoliton} with $R=0$. 
  
  If $\eta>0$, using Lemma \ref{L:channelsoddunsoliton} one can decompose $u_1=\partial_t u(0)$ as
\begin{gather}
\label{Decomp1} 
 u_1=\alpha \tGamma+\beta\Lambda W+\Pi_{L^2_\eta}^{\bot}u_1,\\
 \label{Decomp2} \left\|\Pi_{L^2_{\eta}}^{\bot}u_1\right\|_{L^2_{\eta}}^2 \lesssim E_{out}(\eta):=\lim_{t\to+\infty}\int_{|x|>\eta+t}|\nabla_{t,x}u(t,x)|^2dx.
 \end{gather} 
 Furthermore, writing $\alpha\tGamma=u_1-\beta\Lambda W-\Pi_{L^2_{\eta}}^{\bot}u_1$, we obtain by \eqref{CO40},
 \begin{equation}
  \label{Rem42} 
  \alpha^2\lesssim E_{out}(\eta).
 \end{equation} 
 Combining:
 $$
 \| \alpha \tGamma+\Pi_{L^2_\eta}^{\bot}u_1\|_{L^2_{\eta }}^2 \leq C_{\eta}E_{out}(\eta).
 $$
As $\Pi^{\bot}_{1,\eta}u_1=\Pi^{\bot}_{1,\eta}( \alpha \tGamma+\Pi_{L^2_\eta}^{\bot}u_1)$, the above inequality implies the desired result.
 \end{proof}

\begin{proof}[Proof of Lemma \ref{lem:channelsodd}]

\textbf{Step 1:} reduction.  
We can assume 
$$(u,\partial_tu)_{\restriction t=0}=(0,u_1).$$
Furthermore, we can decompose 
$$u_1=\Pi^{\bot}_{L^2,\lambdabf}u_1+v_1,$$
where $v_1\in \mathrm{span}\left\{ (\Lambda W)_{[\lambda_k]}\right\}_{1\leq k\leq J}$. Letting $v$ be the solution of \eqref{id:linearmultisoliton} with initial data $(0,v_1)$, we see by \eqref{bd:phij} in Lemma \ref{lem:approxresonances} that 
$$\lim_{t\to+\infty}\int_{|x|>|t|} |\nabla_{t,x}v(t,x)|^2dx\lesssim \gamma(\lambdabf)^2\|v_1\|^2_{L^2}\lesssim \gamma(\lambdabf)^2\|u_1\|^2_{L^2}.$$
We are thus reduced to the case where 
\begin{equation}
 \label{reduced}
 \int u_1 (\Lambda W)_{[\lambda_j]}=0, \; j\in \llbracket 1,J\rrbracket.
\end{equation} 

\smallskip

\textbf{Step 2:} \emph{induction.} We prove by induction on $j\in  \llbracket 1,J\rrbracket$ the following property:\\

\emph{Induction Property for $j\in  \llbracket 1,J\rrbracket$}: for any $\eps>0$ and any $\eta>0$, there exists $\gamma^*_j>0$ such that if $0<\gamma(\lambdabf)\leq \gamma^*_j$ then:
\begin{equation}
 \label{Spl10}
 \|u_1\|_{L^2_{\eta\lambda_j}}\leq \eps\|u_1\|_{L^2}+ C_{\eps,\eta} \sqrt{E_{out}},
\end{equation} 
where $E_{out}=\lim_{t\to\infty}\int_{|x|>|t|}|\nabla_{t,x}(u(t,x))|^2dx$ (the constant $C_{\eps,\eta}$ depends also on $J$ but we omit this dependence).\\

We fix $j\in\llbracket 1,J\rrbracket$ and assume that $j=1$, or that $j\geq 2$ and that the induction property holds true for $j-1$, i.e. that for any $\eps_{j-1}>0$, for any $\eta_{j-1}>0$, and for any $\gamma(\lambdabf)$ small enough we have
\begin{equation}
 \label{Spl_induction}
 \|u_1\|_{L^2_{\eta_{j-1}\lambda_{j-1}}}\leq \eps_{j-1}\|u_1\|_{L^2}+ C_{\eps_{j-1},\eta_{j-1}} \sqrt{E_{out}}.
\end{equation} 
To prove the induction property for $j$, we fix $\eps_j>0$ and $\eta_j>0$. Without loss of generality,  we can choose $\eta_j>0$ small enough (depending on $\eps_j$), which we will do later on. Below we use the convention that $\lambda_{j-1}=\infty $ if $j=1$. We let $\eta_{j-1}>0$ and $\eps_{j-1}>0$ to be fixed later on depending on $\eps_j$ and $\eta_j$. In what follows, the dependence of the constants on the parameters $\eps_{j}$, $\eps_{j-1}$, $\eta_j$ and $\eta_{j-1}$ is indicated by subscripts, but the dependence on $J$ is omitted. The notation $C$ thus stands for a generic constant only depending on $J$. We write 
$$\partial_t^2u-\Delta u-2W_{(\lambda_j)}u=2\sum_{ k\neq j} W_{(\lambda_k)}u.$$
By the Claim \ref{Cl:rescaled_unsoliton}, 
\begin{equation}
 \label{Spl20}
 \left\|\Pi^{\bot}_{\lambda_j,\eta_j \lambda_j}u_1\right\|_{L^2_{\eta_j\lambda_j}}\leq C_{\eta_j} \sum_{k \neq j} \left\|\indic_{\{|x|>\eta_j\lambda_j+|t|\}}W_{(\lambda_k)}u\right\|_{L^1L^2}+C_{\eta_j}\sqrt{E_{out}}.
\end{equation} 
If $k>j$, 
\begin{multline*}
 \left\|\indic_{\{|x|>\eta_j\lambda_j+|t|\}}W_{(\lambda_k)}u\right\|_{L^1L^2}\leq \left\|\indic_{\{|x|>\eta_j\lambda_j+|t|\}}W_{(\lambda_k)}\right\|_{L^2L^4}\left\|\indic_{\{|x|>\eta_j\lambda_j+|t|\}}u\right\|_{L^2L^4}\\
 \lesssim \left\|\indic_{\{|x|>\eta_j\lambda_j+|t|\}}W_{(\lambda_k)}\right\|_{L^2L^4}\|u_1\|_{L^2_{\eta_j\lambda_j}}
\end{multline*}
by Lemma \ref{lem:linearmultiout}.
We have by a direct computation
$$\left\|\indic_{\{|x|>\eta_j \lambda_j+|t|\}}W_{(\lambda_k)}\right\|_{L^2L^4}\lesssim \left\|\indic_{\{|x|>\eta_j\lambda_j+|t|\}}\frac{\lambda_k^2}{r^4} \right\|_{L^2L^4}\lesssim \frac{\lambda_k^2}{\eta^2_j\lambda_j^2}\lesssim \frac{\gamma(\lambdabf)^2}{\eta_j^2},  $$
and thus, 
\begin{equation}
\label{Spl21}
  \forall k\in \llbracket j+1,J\rrbracket,\quad \left\|\indic_{\{|x|>\eta_j\lambda_j+|t|\}}W_{(\lambda_k)}u\right\|_{L^1L^2}\lesssim \frac{\gamma(\lambdabf)^2}{\eta^2_j}\| u_1\|_{L^2}.
\end{equation} 
If $k<j$, by H\"older,
\begin{multline}
\label{Spl22}
\left\|\indic_{\{|x|>\eta_j\lambda_j+|t|\}}W_{(\lambda_k)}u\right\|_{L^1L^2}\leq
\left\|\indic_{\{|t|<|x|<\eta_{j-1}\lambda_{j-1}+|t|\}}W_{(\lambda_k)}\right\|_{L^2L^4}\|\indic_{\{|x|>|t|\}}u\|_{L^2L^4}\\
+\left\|\indic_{\{|x|>\eta_{j-1}\lambda_{j-1}+|t|\}}W_{(\lambda_k)}\right\|_{L^2L^4}\|\indic_{\{|x|>\eta_{j-1}\lambda_{j-1}+|t|\}}u\|_{L^2L^4}.
\end{multline}
We have since $\lambda_{j-1}\leq \lambda_k$
$$
\left\|\indic_{\{|t|<|x|<\eta_{j-1}\lambda_{j-1}+|t|\}}W_{(\lambda_k)}\right\|_{L^2L^4}=\left\|\indic_{\{|t|<|x|<\eta_{j-1}\frac{\lambda_{j-1}}{\lambda_k}+|t|\}}W\right\|_{L^2L^4}=o_{\eta_{j-1} \to 0}(1), 
$$
uniformly for $1\leq k\leq j-1 $, and, by Lemma \ref{lem:linearmultiout},
$$
\|\indic_{\{|x|>|t|\}}u\|_{L^2L^4} \lesssim  \|u_1\|_{L^2}. 
$$
Combining, we obtain for all $1\leq k \leq j-1$:
\begin{equation}
\label{cond_eta}
\left\|\indic_{\{|t|<|x|<\eta_{j-1}\lambda_{j-1}+|t|\}}W_{(\lambda_k)}\right\|_{L^2L^4}\|\indic_{\{|x|>|t|\}}u\|_{L^2L^4} 
=o_{\eta_{j-1}\to 0}( \| u_1\|_{L^2} )
\end{equation}  
Next, by Lemma \ref{lem:linearmultiout}, $\|\indic_{\{|x|>\eta_{j-1}\lambda_{j-1}+|t|\}}u\|_{L^2L^4}\leq C\|u_1\|_{L^2_{\eta_{j-1}\lambda_{j-1}}}$. Combining with the induction hypothesis \eqref{Spl_induction}, we obtain
\begin{equation}
 \label{Spl31}
 \|\indic_{\{|x|>\eta_{j-1}\lambda_{j-1}+|t|\}}u\|_{L^2L^4}\leq C \left( \eps_{j-1}\|u_1\|_{L^2}+C_{\eps_{j-1},\eta_{j-1}}\sqrt{E_{out}}\right).
\end{equation}
Combining the inequalities \eqref{Spl22},\eqref{cond_eta}, and \eqref{Spl31}, we deduce that
\begin{equation}
 \label{Spl32}
 \forall k\in \llbracket 1,j-1\rrbracket,\quad \left\|\indic_{\{|x|>\eta_j\lambda_j+|t|\}}W_{(\lambda_k)}u\right\|_{L^1L^2}\leq  \left(C\eps_{j-1}+o_{\eta_{j-1}}(1)\right) \|u_1\|_{L^2}+C_{\eps_{j-1},\eta_{j-1}}\sqrt{E_{out}},
 \end{equation} 
 where $o_{\eta_{j-1}}(1)\to 0$ as $\eta_{j-1}\to 0$. 
 By \eqref{Spl20}, \eqref{Spl21} and \eqref{Spl32}, we obtain
 \begin{align}
  \nonumber
  \left\|\Pi^{\bot}_{\lambda_j,\eta_j \lambda_j}u_1\right\|_{L^2_{\eta\lambda_j}} & \leq  C_{\eta_j} \left( \eps_{j-1}+o_{\eta_{j-1}}(1)+\frac{\gamma^2(\lambdabf)}{\eta_j^2}\right) \|u_1\|_{L^2}+ C_{\eps_j,\eta_j,\eps_{j-1},\eta_{j-1}} \sqrt{E_{out}} \\
\label{Spl33}  & \leq \frac{\eps_j}{2} \|u_1\|_{L^2}+ C_{\eps_j,\eta_j,\eps_{j-1},\eta_{j-1}} \sqrt{E_{out}} 
 \end{align} 
 where the last inequality holds for any $\eta_j,\eps_j>0$, provided $\eps_{j-1}$ and $\eta_{j-1}$ are chosen small enough depending on $\eps_j$ and $\eta_j$, and then $\gamma$ is taken small enough.
 
 Next, we observe that (since $u_1\bot (\Lambda W)_{[\lambda_j]}$ in $L^2$),
 \begin{multline*}
  \left|\int_{|x|>\eta_j\lambda_j} u_1(\Lambda W)_{[\lambda_j]}\right|=\left|\int_{|x|<\eta_j\lambda_j} u_1(\Lambda W)_{[\lambda_j]}\right|\leq \|u_1\|_{L^2}\left\| (\Lambda W)_{[\lambda_j]}\right\|_{L^2(\{|x|<\eta_j\lambda_j\})}\\
  \leq \|u_1\|_{L^2} \|\Lambda W\|_{L^2(\{|x|<\eta_j\})}=o_{\eta_j}(1)\|u_1\|_{L^2},
 \end{multline*}
Together with \eqref{Spl33} we obtain the desired conclusion \eqref{Spl10}.

\smallskip

\noindent\textbf{Step 3}. \emph{Conclusion of the proof.}

Using the conclusion of Step 2, and the same argument as in Step 2 with $j=J$ and $\eta_J=0$, we obtain that for any $\eps_J>0$, there exists a constant $C_{\eps_J}$ such that for $\gamma(\lambdabf)$ small enough,
 \begin{equation*}
 \|u_1\|_{L^2}\leq \eps_J \|u_1\|_{L^2}+ C_{\eps_J} \sqrt{E_{out}}.
 \end{equation*}
Choosing $0<\eps_J<1$, this implies that $\|u_1\|_{L^2}\leq \sqrt{E_{out}}$ which is the desired result.

\end{proof}

\subsection{Channels of energy around a multisoliton for even in time solutions}
\label{Sub:even}

In this subsection we prove:
\begin{lemma}
\label{L:channelseven}
For any $J\in \mathbb N$, there exist $\gamma^*,C>0$ such that for all $\lambdabf \in \Lambda_J$ with $\gamma(\lambdabf)\leq \gamma^*$, any radially symmetric solution $u$ of \eqref{id:linearmultisoliton} satisfies:
\begin{equation}
\label{bound_even_sol}
\| \nabla \Pi^\perp_{\dot{H}^1,\lambdabf} u_0 \|_{Z_{-3,\lambdabf}}^2 \leq C \left(\sum_{\pm } \lim_{t\rightarrow \pm \infty}  \int_{r\geq |t|} |\nabla_{t,x}u|^2dx+\gamma(\lambdabf)^4|\log \gamma(\lambdabf)|^2 \| u_0\|_{Z_{-3,\lambdabf}}^2\right). 
\end{equation} 
\end{lemma}

\begin{proof}
\textbf{Step 0}.  \emph{Preliminary reduction}. 

In all the proof, to lighten notation, we denote 
$$\gamma=\gamma(\lambdabf) \quad\text{and}\quad E_{out}=\sum_{\pm } \lim_{t\rightarrow \pm \infty}  \int_{r\geq |t|} |\nabla_{t,x}u|^2dx.$$
Replacing $u(t,x)$ by $u(t,x)+u(-t,x)$, we see that we can assume $u_1=0$.  

We also claim that it suffices to prove the estimate
\be \label{odd:bd:reducedresult}
\| \nabla \Pi^\perp_{\dot{H}^1,\lambdabf} u_0 \|_{Z_{-3,\lambdabf}} \lesssim \sqrt{E_{out}}
\ee
when 
\be \label{odd:id:assumption}
u_0=\Pi^\perp_{\dot{H}^1,\lambdabf} u_0
\ee
Indeed, assume that we have proved \eqref{bound_even_sol} for all solutions of \eqref{id:linearmultisoliton} with initial data $(u_0,0)$ satisfying \eqref{odd:id:assumption}. Let then $u$ be a solution of \eqref{id:linearmultisoliton} with initial data $(u_0,0)$ with $u_0\in \dot{H}^1$ and write $u_0=v_0+w_0$ with $w_0=\Pi^{\bot}_{\dot{H}^1,\lambdabf}u_0$. Let $v$ and $w$ be the corresponding solutions of \eqref{id:linearmultisoliton}. By our assumption,
\begin{equation}
 \label{St0:1}
 \|\nabla w_0\|_{Z_{-3,\lambdabf}}^2\lesssim \lim_{t\to\infty} \int_{|x|>|t|}|\nabla_{t,x}w(t,x)|^2dx.
\end{equation} 
We claim
\begin{equation}
  \label{St0:2}
 \lim_{t\to\infty} \int_{|x|>|t|}|\nabla_{t,x}v(t,x)|^2dx\lesssim \gamma^4|\log \gamma|^2 \|\nabla v_0\|^2_{Z_{-3,\lambdabf}}.
\end{equation}
By \eqref{St0:1} and \eqref{St0:2}, we have 
\begin{multline*}
\|\nabla w_0\|_{Z_{-3,\lambdabf}}^2\lesssim E_{out}+\lim_{t\to\infty} \int_{|x|>|t|}|\nabla_{t,x}v(t,x)|^2dx
\lesssim E_{out}+\gamma^4|\log \gamma|^2\|\nabla v_0\|_{Z_{-3,\lambdabf}}^2\\\lesssim E_{out}+\gamma^4|\log \gamma|^2\left( \|\nabla u_0\|_{Z_{-3,\lambdabf}}^2+\|\nabla w_0\|_{Z_{-3,\lambdabf}}^2 \right),
\end{multline*}
which yields the desired conclusion, using the smallness of $\gamma$. 

To conclude Step 0, it remains to prove \eqref{St0:2}. Denoting $v_0=\sum_{j=1}^J \alpha_j (\Lambda W)_{(\lambda_j)}$, we see that \eqref{St0:2} is a direct consequence of 
\begin{gather}
 \label{St0:3}
 \|\nabla v_0\|_{Z_{-3,\lambdabf}}\gtrsim \sup_{1\leq j\leq J} |\alpha_j|\\
 \label{St0:4}
 \lim_{t\to\infty}\int_{|x|>|t|} |\nabla_{t,x}v(t,x)|^2dx\lesssim \gamma^4|\log \gamma|^2 \sup_{1\leq j\leq J}\alpha_j^2.
\end{gather}
The inequality \eqref{St0:3} is an easy consequence of the definition of the norm $Z_{-3,\lambdabf}$ and the fact that $\gamma$ is small. The estimate \eqref{St0:4} follows from \eqref{bd:psij} in Lemma \ref{lem:approxresonances}.\\

\textbf{Step 1}. \emph{Formulating the induction property}. We introduce the notation for $1\leq j \leq J$:
$$
R_j^+=\sqrt{\lambda_j\lambda_{j+1}}, \qquad R_j^{-}=\frac{\lambda_{j+1}}{\gamma^{\frac 14}}.
$$
Note that $R_j^-\leq \gamma^{\frac 14} R_j^+$. For $1\leq j\leq J$ we define $\lambdabf_j=(\lambda_1,...,\lambda_j)$ and, for any $\lambda>0$:
$$
\| f\|_{Z_{-3,\lambda}}=\|  f_{[\frac{1}{\lambda}]}\|_{Z_{-3}}.
$$
Observe that for any $1\leq j \leq J$:
\be \label{bd:channelsevenmultiinter}
\| f\|_{Z_{-3,\lambda_j}}  \leq \| f\|_{Z_{-3,\lambdabf}},
\ee
and that if moreover $\textup{supp}(f)\subset [R_{j}^+,\infty)$ then:
\be \label{bd:channelsevenmultiinter2}
 \| f\|_{Z_{-3,\lambdabf}} \lesssim \| f\|_{Z_{-3,\lambdabf_j}}.
\ee
We fix $\alpha = 1/16$. We will prove by induction on $j=1,...,J$ the following property:\\

\emph{Induction Property for $j\in  \llbracket 1,J\rrbracket$}. There holds:
\be \label{odd:bd:ionductionprop}
\| \indic_{\{r\geq R_j^+\}} \nabla u_{0} \|_{Z_{-3,\lambdabf_j}}\lesssim \sqrt{E_{out}}+\gamma^\alpha \| \nabla u_0\|_{Z_{-3,\lambdabf}}.
\ee
Then, for $j=J$, \eqref{odd:bd:ionductionprop} implies \eqref{odd:bd:reducedresult} as $R_J^+=0$ and the proof of the Lemma is over.\\

\textbf{Step 2}. \emph{Proof of the induction property for $2\leq j \leq J-1$}. We assume that the induction property \eqref{odd:bd:ionductionprop} holds true for $j-1$ for some $2\leq j \leq J-1$ and prove in this step that it is true for $j$ as well. We decompose the solution as:
\be \label{odd:id:decompositioninner}
u=v_{j}+\tilde v_j,
\ee
where
$$
\pa_t^2 v_j-\Delta v_j+V_{(\lambda_j)}v_j=0, \qquad (v_j(0),\pa_t v_j(0))=(u_{0},0)
$$
and
$$
\pa_t^2 \tilde v_j-\Delta \tilde v_j+V_{\lambdabf}\tilde v_j=F(v_j)=-\sum_{i\neq j}V_{(\lambda_i)}v_j, \qquad (\tilde v_j(0),\pa_t \tilde v_j(0))=(0,0).
$$
Let $\chi$ denote a smooth cut-off function with $\chi(r)=1$ for $r\leq 1$ and $\chi(r)=0$ for $r\geq 2$. To estimate $F(v_j)$, we decompose in three pieces $v_j^m$ for $m=1,2,3$ that each solve
$$
\pa_t^2 v_j^m-\Delta v_j^m+V_{(\lambda_j)}v_j^m=0,
$$
with data
$$
v_j^1(0)=\chi_{R_{j-1}^+}(1-\chi_{R_{j}^-/48})u_0, \qquad v_j^2(0)= (1-\chi_{R_{j-1}^+})u_0, \qquad v_j^3(0)= \chi_{R_{j}^-/48}u_0,
$$
(where $\chi_R(r)=\chi(r/R)$),
so that $v_j=v_j^1+v_j^2+v_j^3$, and hence
\be \label{odd:id:decompositionF}
F(v_j)=F(v_j^1)+F(v_j^2)+F(v_j^3).
\ee
By finite speed of propagation, $v_j^3(t,r)=0$ for $r\geq R_j^-/24+|t|$ and hence:
$$
 \left\| \indic_{\{r\geq R_j^-/24+|t|\}}F(v_j^3)\right\|_{L^1L^2}=0.
$$
We claim the following estimates, whose proofs are relegated to Step 4: 
\be \label{odd:bd:F}
 \left\| \indic_{\{r\geq R_j^-/24+|t|\}}F(v_j^1)\right\|_{L^1L^2}\lesssim \gamma^\alpha \| \nabla u_0\|_{Z_{-3,\lambdabf}},
\ee
and 
\be \label{odd:bd:F2}
\left\|\indic_{\{r\geq R_j^-/24+|t|\}} F(v_j^2)\right\|_{L^1L^2}\lesssim \sqrt{E_{out}}+\gamma^\alpha \| \nabla u_0\|_{Z_{-3,\lambdabf}}.
\ee
Hence, applying Lemma \ref{lem:linearmultiout} we infer:
$$
\sup_{t} \| (\tilde v_j(t),\pa_t \tilde v_j(t)) \|_{\mathcal H_{R_j^-+|t|}}\lesssim  \sqrt{E_{out}}+\gamma^\alpha \| \nabla u_0\|_{Z_{-3,\lambdabf}}
$$
and hence because of \eqref{odd:id:decompositioninner}:
$$
\sum_{\pm} \lim_{t\to \infty} \int_{|x|\geq R_{j}^-/24+|t|} |\nabla_{t,x} v_j(t)|^2dx \lesssim E_{out}+\gamma^{2\alpha }\| \nabla u_0\|_{Z_{-3,\lambdabf}}^2.
$$
Therefore, applying Lemma \ref{L:channelsevenunsoliton2} to $v_j$, we infer that there exist two constants $c_j,d_j$ such that:
\be \label{odd:id:decompositionu0lambdaj}
u_0=c_j \Upsilon_{(\lambda_j)}+d_j(\Lambda W)_{(\lambda_j)}+\bar u_{0,j}
\ee
with
\be \label{odd:bd:barw0j}
\| \indic_{\{r\geq R_j^-\}}\nabla \bar u_{0,j}\|_{Z_{-3,\lambda_j}}\lesssim \sqrt{E_{out}}+\gamma^\alpha \| \nabla u_0\|_{Z_{-3,\lambdabf}}.
\ee
As a consequence
\begin{equation*}
 c_j^2\int_{R_j^-}^{2R_j^-} |\nabla \Upsilon_{(\lambda_j)}|^2r^5dr\lesssim \int_{R_j^-}^{2R_j^-} |\nabla u_0|^2r^5dr+\int_{R_j^-}^{2R_j^-} |\nabla \bar u_{0,j}|^2r^5dr+d_j^2\int_{R_j^-}^{2R_j^-} |\nabla (\Lambda W)_{(\lambda_j)}|^2r^5dr.
\end{equation*}
Using the fact that for $r\approx R_j^-$ we have $|\nabla (\Lambda W)_{(\lambda_j)}|\lesssim \frac{1}{(R_j^-)^3} \frac{\lambda_{j+1}^4}{\gamma\lambda_j^4}$ and $|\nabla \Upsilon_{(\lambda_j)}|\gtrsim \frac{\sqrt{\gamma}}{(R_j^-)^{3}}\frac{\lambda_j^2}{\lambda_{j+1}^2}$ we get:
\begin{multline*}
 c_j^2\gamma \frac{\lambda_j^4}{\lambda_{j+1}^4}\lesssim \langle \log \gamma\rangle^2 \| \nabla u_0\|^2_{Z_{-3,\lambda_{\bf}}}+\left\langle \log \frac{\lambda_{j+1}}{\lambda_j\gamma^{1/4}}\right\rangle^2 \| \indic_{\{r\geq R_j^-\}}\nabla \bar u_{0,j}\|_{Z_{-3,\lambda_j}}^2
 +d_j^2\left( \frac{\lambda_{j+1}^4}{\gamma\lambda_j^4} \right)^2.
\end{multline*}
Combining with the bound \eqref{odd:bd:barw0j}, we deduce
\be \label{odd:bd:dj}
|c_j|\lesssim \frac{1}{\sqrt{\gamma}} \frac{\lambda_{j+1}^2}{\lambda_j^2} \left|\log \left(\frac{1}{\sqrt{\gamma}} \frac{\lambda_{j+1}^2}{\lambda_j^2}\right)\right| \left(\frac{1}{\gamma} \frac{\lambda_{j+1}^4}{\lambda_j^4} |d_j|+ \| \nabla u_0\|_{Z_{-3,\lambdabf}} +\sqrt{E_{out}} \right).
\ee
As $u_0=\Pi^\perp_{\dot{H}^1,\lambdabf}u_0 $ we have $\int \nabla u_0 \nabla (\Lambda W)_{(\lambda_j)}=0$, which, after truncation, using $|\nabla \Lambda W(r)|\lesssim r$ and $R_{j}^-=\frac{\lambda_{j+1}}{\gamma^{\frac14}}\leq \gamma^{\frac 34}\lambda_j$, gives:
$$
\int_{R_j^-\leq |x|} \nabla u_0 \nabla (\Lambda W)_{(\lambda_j)}=\int_{ |x|\leq R_j^-} \nabla u_0 \nabla (\Lambda W)_{(\lambda_j)}=O(\gamma^3|\log \gamma| \| \nabla u_0\|_{Z_{-3,\lambdabf}}).
$$
On the other hand, computing the left-hand side above using \eqref{odd:id:decompositioninner}, $ \int_{|x|\leq \frac{R_j^-}{\lambda_j}} |\nabla \Lambda W|^2\lesssim \frac{R_j^{-8}}{\lambda_j^8}\lesssim \gamma^6$ and the bound \eqref{odd:bd:barw0j} shows:
$$
\int_{R_j^-\leq |x|} \nabla u_0 \nabla (\Lambda W)_{(\lambda_j)}= d_j \int |\nabla \Lambda W|^2 (1+O(\gamma^6))+O\left(|c_j|+ \sqrt{E_{out}}+\gamma^\alpha \| \nabla u_0\|_{Z_{-3,\lambdabf}}\right).
$$
Combining the two identities above shows:
\be \label{odd:bd:cj-1}
| d_j|\lesssim |c_j|+ \sqrt{E_{out}}+\gamma^\alpha\| \nabla u_0\|_{Z_{-3,\lambdabf}}.
\ee
Combining the two inequalities \eqref{odd:bd:dj} and \eqref{odd:bd:cj-1} shows:
$$
| d_j|\lesssim  \sqrt{E_{out}}+\gamma^{\alpha}\| \nabla u_0\|_{Z_{-3,\lambdabf}},
$$
and
$$
|c_j|\lesssim \frac{1}{\sqrt{\gamma}} \frac{\lambda_{j+1}^2}{\lambda_j^2} \log \left(\frac{1}{\sqrt{\gamma}} \frac{\lambda_{j+1}^2}{\lambda_j^2}\right) \left( \| \nabla u_0\|_{Z_{-3,\lambdabf}} +\sqrt{E_{out}} \right)\lesssim \frac{1}{\gamma^{\frac 38}} \frac{\lambda_{j+1}^{\frac 32}}{\lambda_j^{\frac 32}}  \left( \| \nabla u_0\|_{Z_{-3,\lambdabf}} +\sqrt{E_{out}} \right).
$$
As a result, since $\| \nabla (\Lambda W)_{(\lambda_j)}\|_{Z_{-3,\lambda_j}}\lesssim 1$:
\be \label{odd:bd:dj2}
\| d_j \nabla (\Lambda W)_{(\lambda_j)}\|_{Z_{-3,\lambda_j}}\lesssim   \sqrt{E_{out}}+\gamma^\alpha \| \nabla u_0\|_{Z_{-3,\lambdabf}},
\ee
and, since for $r\geq R_j^+$ we have $|\nabla \Upsilon_{(\lambda_j)}|\lesssim \frac{1}{r^3} \frac{\lambda_j^2}{R_j^{+2}}=\frac{1}{r^3}\frac{\lambda_j}{\lambda_{j+1}}$, we infer:
\begin{multline} \label{odd:bd:cj2}
\| \indic_{\{r\geq R_j^+\}}c_j \nabla (\Upsilon W)_{(\lambda_j)}\|_{Z_{-3,\lambda_j}}\\
\lesssim |c_j|\frac{\lambda_j}{\lambda_{j+1}}\lesssim \frac{1}{\gamma^{\frac38}} \frac{\lambda_{j+1}^{\frac 12}}{\lambda_j^{\frac 12}}  \left( \| \nabla u_0\|_{Z_{-3,\lambdabf}} +\sqrt{E_{out}} \right)\lesssim \gamma^{\frac 18} \| \nabla u_0\|_{Z_{-3,\lambdabf}} +\gamma^{\frac 18}\sqrt{E_{out}}.
\end{multline}
Injecting \eqref{odd:bd:barw0j}, \eqref{odd:bd:dj2} and \eqref{odd:bd:cj2} in \eqref{odd:id:decompositionu0lambdaj} shows:
$$
\big\| \indic_{\{r\geq R_j^+\}} \nabla u_{0} \big\|_{Z_{-3, \lambda_j}}\lesssim \sqrt{E_{out}}+\gamma^\alpha \| \nabla u_0\|_{Z_{-3,\lambdabf}}.
$$
Combined with the induction hypothesis $\big\| \indic_{\{r\geq R_{j-1}^+\}} \nabla u_{0} \big\|_{Z_{-3,\lambdabf_{j-1}}}\lesssim \sqrt{E_{out}}+\gamma^\alpha \| \nabla u_0\|_{Z_{-3,\lambdabf}}$, and the inequality
$$ \big\|\indic_{\{r>R_j^+\}}f\big\|_{Z_{-3,\lambdabf_{j}}}\lesssim \big\|\indic_{\{r>R_{j-1}^+\}} f\big\|_{Z_{-3,\lambdabf_{j-1}}}+\big\|\indic_{\{r>R_j^+\}} f\big\|_{Z_{-3,\lambda_j}}$$
we obtain
\eqref{odd:bd:ionductionprop} for $j$.\\

\textbf{Step 3}. \emph{Proof of the induction property for $j=1$ and $j=J$}. For $j=1$, the proof is the exact same one as in Step 2, but with the simplification that $v_j^2=0$ in the decomposition \eqref{odd:id:decompositionF}, so that all computations and estimates of Step 2 are still valid for $j=1$ with the convention that $R_{j-1}^-=R_{j-1}^+=\infty$, and hence the result holds true for $j=1$.

For $j=J$, the proof is again the exact same one, but with the simplification that $c_j=0$ in \eqref{odd:id:decompositionu0lambdaj} since Lemma \ref{L:channelsevenunsoliton2} is applied for $v_J$ for $r\geq R_{J}^-=0$. All computations and estimates of Step 2 are then still valid for $j=J$ with the convention that $R_{J}^-=R_{J}^+=0$, and hence the result holds true for $j=J$.\\

\textbf{Step 4}. \emph{Estimating the error terms}. In this step we prove \eqref{odd:bd:F} and \eqref{odd:bd:F2}.

We start with the proof of \eqref{odd:bd:F}. Let $1\leq i \leq J$ with $i\neq j$. By \eqref{localisingZ-32} and \eqref{bd:channelsevenmultiinter} we have:
$$
\| \nabla v_1^j(0)\|_{Z_{-3,\lambda_i}}\leq \| \nabla v_1^j(0)\|_{Z_{-3,\lambdabf}}\lesssim \| \nabla u_0 \|_{Z_{-3,\lambdabf}}.
$$
If $i>j$, then notice that $v^1_j(0,r)=0$ for $r\leq  \frac{\lambda_i}{48\gamma^{1/4}}\leq \frac{R_j^-}{48}$. Applying \eqref{interactionsolitonZ-33} and a scaling argument, we obtain using the above inequality:
$$
\| \indic_{\{r\geq |t|\}} V_{(\lambda_i)}v^1_j\|_{L^1L^2}\lesssim \gamma^{\frac 14} \| \nabla v_j^1\|_{Z_{-3,\lambda_i}}\lesssim \gamma^{\frac 14} \| \nabla u_0 \|_{Z_{-3,\lambdabf}}.
$$
If $i<j$, then notice that $v^1_j(0,r)=0$ for $r\geq 2\sqrt{\gamma} \lambda_i \geq 2R_{j-1}^+$. Applying \eqref{interactionsolitonZ-32}, we obtain in this case:
$$
\| \indic_{\{r\geq |t|\}} V_{(\lambda_i)}v^1_j\|_{L^1L^2}\lesssim \gamma^{\frac{1}{16}} \| \nabla u_0 \|_{Z_{-3,\lambdabf}}.
$$
The two inequalities above imply \eqref{odd:bd:F}.

We now prove \eqref{odd:bd:F2}. Let $1\leq i \leq J$ with $i\neq j$. By \eqref{localisingZ-32} and the induction hypothesis \eqref{odd:bd:ionductionprop} at $j-1$:
$$
\| \nabla v_j^2(0)\|_{Z_{-3,\lambdabf_{j-1}}}\lesssim \| \indic_{\{r\geq R_{j-1}^+\}} \nabla u_0\|_{Z_{-3,\lambdabf_{j-1}}}\lesssim  \sqrt{E_{out}}+\gamma^\alpha \| \nabla u_0\|_{Z_{-3,\lambdabf}}.
$$
As $v_j^2(0,r)=0$ for $r\leq R_{j-1}^+$, by \eqref{bd:channelsevenmultiinter} and \eqref{bd:channelsevenmultiinter2}:
$$
\| \nabla v_j^2(0)\|_{Z_{-3,\lambda_i}}\leq \| \nabla v_j^2(0)\|_{Z_{-3,\lambdabf}} \lesssim  \| \nabla v_j^2(0)\|_{Z_{-3,\lambdabf_{j-1}}}.
$$
Combining the two inequalities above,
$$
\| \nabla v_j^2(0)\|_{Z_{-3,\lambda_i}}\lesssim  \sqrt{E_{out}}+\gamma^\alpha \| \nabla u_0\|_{Z_{-3,\lambdabf}}.
$$
Applying \eqref{interactionsolitonZ-31}, we obtain:
$$
\| \indic_{\{r>|t|\}} V_{(\lambda_i)} v_j^2\|_{L^1L^2}\lesssim  \sqrt{E_{out}}+\gamma^\alpha \| \nabla u_0\|_{Z_{-3,\lambdabf}}.
$$
This implies \eqref{odd:bd:F2}, what ends the proof of the Lemma.

\end{proof}

\section{Channels of energy estimate around the ground state in $8$ dimensions} \label{sec:8d}

In this section we consider the linearised wave equation in eight dimensions
\be \label{8dLW}
\pa_t^2u_L-\Delta u_L+Vu_L=0, \qquad \vec u_L(0)\in \mathcal H(\mathbb R^8),
\ee
where $V=-\frac 53 W^{\frac 23}$ with $W=\left(1+\frac{|x|^2}{48} \right)^{-3}$. The function $\Lambda W=x.\nabla W +3 W$ belongs to the kernel of $-\Delta +V$. We claim an analogous channels of energy estimate to \eqref{bd:channelunsoliton}:

\begin{theorem} \label{th:channelsunsoliton8d}
There exists $C>0$ such that any radial solution $u_L$ of \eqref{8dLW} satisfies:
$$
\big\| \Pi_{L^2}^\perp u_1\big\|_{Z_{-4}}+\big\| \Pi_{\dot H^1}^\perp u_0 \big\|_{\dot H^1} \leq C \sqrt{E_{\textup{out}}}.
$$

\end{theorem}

We believe this result can be extended to all dimensions $N\equiv 4\mod 4$ up to a technical refinement of the proof. We only sketch the proof of Theorem \ref{th:channelsunsoliton8d}, as a consequence of Lemmas \ref{lem:iteratedkernel}, \ref{8d:L:estimateufromutunsoliton} and \ref{8d:lem:estimateufromutunsoliton2}. We highlight only the main differences with that, similar, of Theorem \ref{th:channelsunsoliton}. We believe an extension for an estimate around a multisoliton like \eqref{bd:channelsmultisoliton} could be proved along the lines of Section \ref{subsec:channels_multi}.

\subsection{Channels of energy for even in time solutions}

The main difference between six and eight dimensions is that the analogue of Proposition \ref{P:CK1.4} only holds true for even in time functions.

\begin{proposition}[Channels in 8d]
 Let $u\in \RR\times \RR^8$, solve $\pa_t^2 u-\Delta u=f$ with $\vec u(0)\in \mathcal H$ and $f\in L^1L^2$. Fix $R>0$ and write $u_0=\frac{c_1}{r^6}+\frac{c_2}{r^4}+u_0^{\bot}$, where $\int_{R}^{\infty} \pa_r u_0^{\bot} r^{-7+2k} r^7dr=0$ for $k=0,1$. Then,
\begin{equation} \label{bd:freewaveschannels8d}
 \left\|u_0^{\bot}\right\|_{\dot H^1_R}\lesssim\sqrt{E_{\textup{out}}}+\|f\|_{L^1_tL^2_r\left( \Ce_{0,R} \right)}.
\end{equation}
 \end{proposition}

\begin{proof}

The estimate for $f=0$ is proved in \cite{LiShenWei21P}. The extension for nonzero $f$ is similar to the proof of Proposition \ref{P:CK1.4}.

\end{proof}

Counter examples to a channel of energy estimate for Equation \eqref{8dLW} are built on the generalised kernel of $-\Delta+V$. Namely, in eight dimensions another radial zero of $-\Delta +V$ is given by a function $\Gamma $ with
\begin{equation}\label{8d:id:asymptoticGamma}
\Gamma(r)\sim cr^{-6}\quad \mbox{as }r\rightarrow 0, \qquad \mbox{ and }\qquad \Gamma(r)\sim c' \quad \mbox{as }r\rightarrow \infty.
\end{equation}

We define $T^\infty_{0}=\Lambda W$, $T^0_0=\Gamma$, and $T^\infty_1$ and $T^0_1$ by the following Lemma:

\begin{lemma} \label{8d:lem:generalisedkernel}

There exist two functions $T^\infty_1$ and $T^0_1$ solving:
$$
(-\Delta +V)T^\infty_1=-T^\infty_0\qquad  \mbox{and}\qquad (-\Delta +V) T^0_1=-T^0_0,
$$
that satisfy for $c^\infty_{1,0},c^\infty_{1,\infty},c^0_{1,0},c^0_{1,\infty},e^0_0, \tilde c^0_{0,0}\neq 0$:
\begin{align*}
&T^\infty_1(r)\underset{r\to \infty}{\sim} c^\infty_{1,\infty} r^{-4}, \qquad T^\infty_1 (r) \underset{r\to 0}{\sim} c^\infty_{1,0} r^{-6},\\
&T^0_1(r)\underset{r\to \infty}{\sim} c^0_{1,\infty}r^2 , \qquad T^0_1 (r) \underset{r\to 0}{\sim} c^0_{1,0} r^{-4},\\
&\tilde T^0_0(r)=T^0_0-e^0_0T^\infty_1(r)\underset{r\to 0}{\sim} \tilde c^0_{0,0}r^2.
\end{align*}
\end{lemma}

\begin{proof}

It is a standard analysis of the underlying second order ODEs in the radial variable.

\end{proof}

Three explicit solutions of $\pa_t^2u_L-\Delta u_L+Vu_L=0$ such that $\lim_{t\to \pm \infty}  \int_{|x|>R+ |t|} |\nabla_{t,x}u_L(t,x)|^2dx=0$ for any $R>0$ are then given by:
$$
S^\infty_0=T^\infty_0, \qquad S^\infty_1=tT^\infty_0, \qquad S^\infty_2=T^\infty_1+\frac{t^2}{2}T^\infty_0.
$$
These are the only ones:

\begin{lemma} \label{8d:lem:classificationevennonradiative}

Let $R\geq 0$. Assume that $u_L$ is an even in time solution \eqref{8dLW} such that 
$$\lim_{t\to \pm \infty}  \int_{|x|>R+ |t|} |\nabla_{t,x}u_L(t,x)|^2dx=0.$$ Then:
\begin{itemize}
\item[(i)] If $R>0$, there exists $c_0,c_2\in \mathbb R$ such that $u_L=c_0 S^\infty_0+c_2 S^\infty_2$ on $\{|x|>R+|t|\}$.
\item[(ii)] If $R=0$, there exists $c_0\in \mathbb R$ such that $u_L=c_0 S^\infty_0$ for all $x,t$.
\end{itemize}
\end{lemma}

\begin{proof}

The proof works exactly as that of Lemma \ref{L:rigiditylinearoddintime}, using the estimate \eqref{bd:freewaveschannels8d} and the asymptotics of Lemma \ref{8d:lem:generalisedkernel}.

\end{proof}

We next introduce, for $\chi^0$ a cut-off with $\chi^0(r)=1$ for $r\leq 10$ and $\chi^0(r)=0$ for $r\geq 11$:
\begin{equation}\label{8d:id:defpiL2R1}
\Pi^\perp_{\dot H^1_R}=\left\{ \begin{array}{l l} \Pi_{\dot H^1_R} \left( \text{Span}(T^\infty_0,T^\infty_1) \right)^\perp \qquad \mbox{for }R\geq 1, \\ \Pi_{\dot H^1_R} \left( \text{Span}(T^\infty_0,T^\infty_1,\chi^0 T^0_1) \right)^\perp \qquad \mbox{for }0<R< 1,\\ \Pi_{\dot H^1} \left( \text{Span}(T^\infty_0) \right)^\perp \qquad \mbox{for }R=0.
  \end{array} \right.
\end{equation}

\begin{lemma} \label{lem:iteratedkernel}

There exists a constant $C>0$, such that for any even in time solution $u_L$ of \eqref{8dLW} and any $R\geq 0$:
$$
\| \pi_{\dot H^1_R}^\perp u_L(0)\|_{\dot H^1_R}^2 \leq C\lim_{t\to \infty} \int_{|x|>R+|t|} |\nabla_{t,x}u_L|^2dx.
$$

\end{lemma}

\begin{proof}
The proof works exactly as that of Lemma \ref{L:channelsoddunsoliton}, combining the rigidity result of Lemma \ref{8d:lem:classificationevennonradiative}, the channels of energy estimate \eqref{bd:freewaveschannels8d} for $|x|>R+|t|$ with $R>0$, and the following estimate for $R=0$ established in \cite{CoKeSc14}:
$$
 \frac 12 \|u_0\|^2_{\dot H^1}\leq  E_{\textup{out}}^+.
$$

\end{proof}

\subsection{Channels of energy for odd in time solutions}

If $u_L$ is a solution of \eqref{8dLW} with initial data in $\HHH$, we denote:
\be \label{8d:patu2:id:deftildeY}
\| u_L\|_{\tilde Y}=\sup_{t\in \mathbb R, \ R> |t|} \| \Pi^\perp_{\dot H^1_R} u_L(t)\|_{\dot H^1_{R}}.
\ee
We introduce the following operators:
\be \label{8d:patu2:id:defA}
\mathcal A f(r)=\int_r^\infty \rho f(\rho)d\rho, \qquad \mathcal A^{-1}f(r)=-\frac{1}{r}\pa_r f(r).
\ee
They are used to state an averaged estimate in the following Lemma, and to lower the dimension in the proof of the next one.

\begin{lemma}[Estimating $\pa_t u_L$ in average from $u_L$ in $\dot H^1$] \label{8d:L:estimateufromutunsoliton}
Recall the notation \eqref{8d:patu2:id:deftildeY}. Assume that $u_L$ is a radial solution of \eqref{8dLW}. Then:
\be \label{bd:8dboundpatu0H-1}
\| \mathcal A(\Pi^\perp_{L^2} \pa_t u(0))\|_{Z_{-2}} \lesssim  \| u_L\|_{\tilde Y}.
\ee
\end{lemma}

\begin{proof}

The proof is similar to that of Lemma \ref{L:estimateufromutunsoliton}: it consists of a first step that we detail since slightly different, in which one expresses $\pa_t u(0)$ in terms of $u(t,x)$ for $|x|>|t|$ on a dyadic scale. For $k\in \mathbb Z$, applying the fundamental Theorem of calculus twice between $t=0$ and $t=R_k=2^k$, using \eqref{8dLW} and that $u_L$ is even in time, one obtains the identity:
\begin{equation}\label{8d:eq:ellipticinter1unsoliton}
\pa_t u(0)= \frac{1}{R_k} u_L(R_k)+(-\Delta +V)\left( \int_0^{R_k}(1-\frac{t}{R_k})u_L(t)dt\right).
\end{equation}
We let $\tilde u_k(t)=\Pi^\perp_{\dot H^1_{R_k}}u_L(t)$. We decompose for $|t|\leq R_k$:
\begin{equation}\label{8d:id:decompositionutY}
u_L(t)=\alpha^\infty_{0,k}(t)T^\infty_0+\alpha^\infty_{1,k}(t) T^\infty_1 +\alpha^0_{1,k}T^0_1 \chi^0 +\tilde u_k(t).
\end{equation}
with $\alpha_{1,k}^0=0$ by convention for all $k\geq 0$. Using Lemma \ref{8d:lem:generalisedkernel}, and that $\chi^0(r)=1$ for $r\leq 10$, we obtain that for all $r\geq R_k$ if $k\geq 0$, or for all $R_k\leq r \leq 10$ if $k\leq -1$:
$$
(-\Delta +V)\left( \int_0^{R_k}\Big(1-\frac{t}{R_k}\Big)u_L(t)dt\right)= \int_0^{R_k}\Big(1-\frac{t}{R_k}\Big) \Big(- \alpha^\infty_{1,k}(t)T^\infty_0-\alpha^0_{1,k}(t)T^0_0+ (-\Delta+V)\tilde u_k(t)\Big)dt.
$$
Injecting the two above identities in \eqref{8d:eq:ellipticinter1unsoliton}, recalling that $T_0^0=e^0_0T^\infty_1+\tilde T_0^0$ from Lemma \ref{8d:lem:generalisedkernel}, we obtain that there exist constants $c^\infty_{0,k}$, $c^\infty_{1,k}$, $c^0_{0,k}$, $c^0_{1,k}$ with $c^0_{0,k}=0$ and $c^0_{1,k}=0$ for $k\geq 0$ such that, for all $r\geq 0$ if $k\geq 0$, or for all $r \leq 10$ if $k\leq -1$:
\begin{align}
\label{8d:id:utdyadic}\pa_t u(0) &=c^\infty_{0,k}T^\infty_0+c^\infty_{1,k}T^\infty_1+c^0_{0,k} \tilde T^0_0 +c^0_{1,k}T^0_1 \\
\nonumber &\qquad + \frac{1}{R_k} \tilde u_k(R_k)+(-\Delta +V)\left( \int_0^{R_k}(1-\frac{t}{R_k})\tilde u_k(t)dt\right).
\end{align}
The rest of the proof is now very similar to that of Lemma \ref{L:estimateufromutunsoliton}: one first estimate the same way the constants $c^\infty_{1,k}$, $c^\infty_{0,k}$, $c^0_{0,k}$ and $c^0_{1,k}$, and then inject these estimates directly in \eqref{8d:id:utdyadic}, showing \eqref{bd:8dboundpatu0H-1}. This last step is actually simpler since no elliptic equation is involved.

\end{proof}

One then upgrades the averaged $L^2$ bound of Lemma \ref{8d:L:estimateufromutunsoliton} into a weighted $L^2$ bound.

\begin{lemma} \label{8d:lem:estimateufromutunsoliton2}

Assume that $u_L$ solves \eqref{8dLW}. Then:
\be \label{bd:patu8dfinal}
\| \Pi^\perp_{L^2}  \pa_t u(0)\|_{Z_{-4}}\lesssim \| \mathcal A (\Pi^\perp_{L^2} \pa_t u(0))\|_{Z_{-2}} +  \| u\|_{\tilde Y}.
\ee

\end{lemma}

\begin{proof}

The proof is very similar to that of Lemma \ref{lem:estimateufromutunsoliton2}, because one can perform a dimensional reduction leading to a wave equation in $6$ dimensions. Indeed, without loss of generality, assume $\pa_t u(0)=\Pi_{L^2}^\perp \pa_t u_0$, and then consider the new unknown:
$$
 v=\mathcal A \pa_t u.
$$
From the conjugation relation $\Delta_N=\mathcal A^{-1} \Delta_{N-2}\mathcal A$, where $\Delta_N=\pa_{rr}+\frac{N-1}{r}\pa_r$, $v$ is an even in time solution on $\mathbb R^{1+6}$ of
$$
\left\{ \begin{array}{l l} \pa_t^2v-\Delta_6 v+\mathcal A V \mathcal A^{-1} v=0, \\ \vec v(0)=(\mathcal A\pa_t u(0),0 ).  \end{array} \right.
$$
The kernel of the above integro-differential operator is:
\be \label{8d:patu2:id:kernelconjugatedop}
(-\Delta_6 +\mathcal A V \mathcal A^{-1})f=0, \qquad \mbox{for }f\in \textup{Span} \left(\mathcal AT^\infty_0 \ , \ \int_1^r \rho \Gamma (\rho)d\rho \ , \ 1 \right) .
\ee

First, one notices that the initial data is estimated in $Z_{-3}$ by the right-hand of \eqref{bd:patu8dfinal}:
\be \label{8d:patu2:bd:v0}
\| v(0)\|_{Z_{-2}} = \| \mathcal A \pa_t u(0)\|_{Z_{-2}}.
\ee
Second, one remarks that $\pa_tv$, outside the kernel, can be estimated by $u$. Indeed,
\be \label{8d:patu2:id:vt1}
\pa_t v=-\mathcal A (-\Delta_8+V)u(t)=-r\pa_r u(t)-6u(t)-\mathcal A Vu(t),
\ee
We rewrite \eqref{8d:id:decompositionutY} in a unified form for all $k\in \mathbb Z$:
\be \label{8d:patu2:id:decompositionu}
u(t)=  \alpha^\infty_{0,k}(t)T^\infty_0+\alpha^\infty_{1,k}(t) T^\infty_1 +\alpha^0_{1,k}(t)T^0_1\chi^0+\tilde u_k(t),
\ee
where by convention $\alpha^0_{1,k}(t)=0$ for $k\geq 0$. We introduce $\bar T_0^0=(-\Delta_8+V)(T_1^0\chi^0)$, and thus, from \eqref{8d:patu2:id:vt1} and Lemma \ref{8d:lem:generalisedkernel}:
\be \label{8d:patu2:id:decompositionvt}
\pa_t v=  \alpha^\infty_{1,k} \mathcal A T^\infty_0+\alpha^0_{1,k} \mathcal A \bar T^0_0 -\mathcal A (-\Delta_8+V)\tilde u_k(t),
\ee
One can then apply the exact same gain of regularity technique for solutions to the wave equation in $6$ dimensions as in the proof of Lemma \ref{lem:estimateufromutunsoliton2}.

\end{proof}

\begin{appendix}

\section{A counter-example to a linear channel of energy estimate}

\label{ap:counterexample}

We make precise here the statement of Remark \ref{re:channelsunsoliton}. This counter-example has similarities with that of \cite{CoKeSc14}.

\begin{lemma}

There exists a sequence of initial data $(u_{n,0},0)$ with $u_{n,0}\in \dot H^1(\mathbb R^6)$ and:
$$
\sup_{R>0} \| r^{-1} u_{n,0} \|_{L^2(\{R\leq r \leq 2R\})} \to \infty \qquad \mbox{as } n\to \infty
$$
such that the corresponding sequence of solutions $u_{F,n}$ to \eqref{FW} satisfies:
$$
\lim_{t\to  \infty}\int_{|x|\geq t} |\nabla_{t,x} u_{F,n}|^2dx \leq 1, \qquad \forall  n.
$$

\end{lemma}

\begin{proof}

\textbf{Step 1.} \emph{Proof assuming a technical claim}. Let $\chi$ be a smooth cut-off with $\chi(r)=0$ for $r\geq 3$ and $r\leq 1$ that generates a partition of unity:
$$
\sum_{k\in \mathbb Z} \chi_k(r)=1, \qquad \chi_k(r)=\chi(r/R_k), \qquad R_k=2^k.
$$
We introduce for $(c_k)_{k}\in \ell^2(\mathbb Z)$:
$$
v[c_k](t,r)=\sum_k c_k \chi_k \frac{1}{r^2}p_1(\frac tr), \qquad p_1(\sigma)=\sigma^2-\frac 12.
$$
and $\tilde v[c_k]$ the solution to
$$
(\pa_t^2-\Delta)\tilde v =-(\pa_t^2-\Delta) v, \qquad (\tilde v(0),\pa_t \tilde v(0))=(0,0).
$$
We claim that for some $C>0$ independent of $c_k$:
\be \label{counterexample:technicalclaim}
\lim_{t\to \infty} \int_{r\geq t} |\nabla_{t,x} \tilde v(t)|^2dx \leq C \sum_{k} |c_k-c_{k+1}|^2.
\ee
Assuming the claim, the Lemma is proved by choosing $u_{F,n}=v[c_{n,k}]+\tilde v[c_{n,k}]$ with $(c_{n,k})_k\in \ell^2(\ZZ)$ such that $\max_{k\in \mathbb Z} |c_{n,k}|\to \infty$ as $n\to \infty$ while $\sum_{k\in \mathbb Z} |c_{n,k}-c_{n,k+1}|^2\lesssim 1$ uniformly for all $n$. Such an example is given by, for any $0<\delta<1/2$:
$$
c_{n,k}=\left\{\begin{array}{l l} n^{\frac 12 -\delta}-|k-n|^{\frac 12-\delta} \qquad \mbox{for }0\leq k\leq 2n, \\ 0 \qquad \mbox{for } |k-n|\geq n+1.\end{array} \right.
$$

\noindent \textbf{Step 2.} \emph{Proof of the technical claim \eqref{counterexample:technicalclaim}}. We have from the partition of unity and support properties of $\chi$ that $\chi_k(r)=1$ for $\frac 32R_k\leq r \leq 2R_k$ and $\chi_k(r)+\chi_{k+1}(r)=1$ for $2R_k\leq r\leq 3R_k$. Hence $\Delta \chi_k(r)=-\Delta \chi_{k+1}(r)$ and $\pa_r \chi_k=-\pa_r \chi_{k+1}(r)$ for $2R_k\leq r\leq 3R_k$. Using this and $(\pa_t^2-\Delta)(r^{-2}p_1(r/t))=0$ we obtain:
\begin{align*}
& (\pa_t^2-\Delta) v= \sum_k (c_{k+1}-c_k)f_{\{ R_k\}},\\
& f_{\{ R_k\}}(t,r)=\frac{1}{R_k^4}f\left(\frac{t}{R_k},\frac{r}{R_k}\right),\\
& f(t,r)=\left(\Delta \chi \frac{1}{r^2}p_1(\frac tr)+2\pa_r \chi \pa_r \left( \frac{1}{r^2}p_1(\frac tr)\right)\right){\bf 1}(2 \leq r \leq 3).
\end{align*}
Therefore,
\begin{align*}
& \tilde v=-\sum_k (c_{k+1}-c_k)w_{(R_k)}, \\
& w_{(R_k)}(t,r)=\frac{1}{R_k^2}w\left(\frac{t}{R_k},\frac{r}{R_k}\right),\\
& (\pa_t^2-\Delta) w= f, \qquad \mbox{with} \quad (w(0),\pa_t w(0))=(0,0).
\end{align*}
We recall that for any solution to $\pa_t^2 u-\Delta u=F$, with $\vec u(0)\in \mathcal H$, and $F\in L^1L^2(r> |t|)$, there exists a radiation profile $G_+\in L^2([0,\infty))$ satisfying
\begin{equation}\label{gainreg:id:radiation6d}
\lim_{t\rightarrow \infty} \int_{t}^\infty \left|r^{\frac 52}\pa_t u(t,r)-G_+(r-|t|)\right|^2dr =\lim_{t\rightarrow \infty} \int_{t}^\infty \left|r^{\frac 52}\pa_r u(t,r)+ G_+(r-|t|)\right|^2dr =0,
\end{equation}
see \cite{DuKeMe20}. Let $G_+[w]$ be the radiation profile of $w$. Since $\pa_t f\in L^1L^2(r\geq |t|)$, then $G_+[w]\in H^1([0,\infty))$, see again \cite{DuKeMe20}. Due to support considerations and finite speed of propagation, it has compact support on the right: $G_+[w](\rho)=0$ for all $\rho\geq \rho_0$, for some $\rho_0>0$. Moreover the radiation profile of $w_k$ is:
$$
G_+[w_k]=G_{+,R_k}, \qquad G_{+,R_k}(\rho)=\frac{1}{\sqrt{R_k}}G_+[w](\frac{\rho}{R_k}).
$$
Hence the radiation emitted by $\tilde v$ as $t\to \infty$ is:
$$
G_+[\tilde v]=\sum_{k\in \mathbb Z} (c_{k+1}-c_k)G_{+,R_k}.
$$
Applying \eqref{gainreg:id:radiation6d} to $\tilde v$, introducing $d_k=c_{k+1}-c_k$ and using a scaling argument:
\begin{align*}
& \lim_{t\to \infty} \int_{r\geq |t|}|\nabla_{t,x}  \tilde v|^2dx \  = \ 2 \|  G_+[\tilde v]\|_{L^2([0,\infty))}^2\\
&\qquad \qquad =\sum_{k,l} d_kd_l \langle G_{+,R_k}, G_{+,R_{l}}\rangle_{L^2([0,\infty))}=\sum_{k,l} d_kd_l \langle G_{+,R_0}, G_{+,R_{l-k}}\rangle_{L^2([0,\infty))}.
\end{align*}
Let $e_k= \langle G_{+,R_0}, G_{+,R_{k}}\rangle_{L^2([0,\infty))}$. Then, since $G_+[w]$ has compact support and is continuous,
$$
e_k\sim R_k^{\frac 12}G_+[w](0)\int_0^\infty G_+[w] d\rho \quad \mbox{as }k\to -\infty, \qquad e_k\sim R_k^{-\frac 12}G_+[w](0)\int_0^\infty G_+[w] d\rho \quad \mbox{as }k\to \infty.
$$
Hence $e_k\in \ell^1$, and, by Young inequality for convolution followed by Cauchy-Schwarz:
$$
\lim_{t\to \infty} \int_{r\geq |t|}|\nabla_{t,x}  \tilde v|^2dx = \sum_{k} d_k (d_{\cdot}*e_{-\cdot})_k\lesssim \sum_k |d_k|^2.
$$
This proves the claim \eqref{counterexample:technicalclaim}.

\end{proof}

\section{A few estimates}
\label{A:estimates}
\begin{lemma}
If $R<R'<\lambda$, 
\begin{equation}
 \label{estim5}
 \Big\|W_{(\lambda)} \indic_{\{R+|t|<|x|<R'+|t|\}}\Big\|_{L^2L^4}\lesssim \left(  \frac{R'-R}{\lambda}\right)^{1/4}.
\end{equation}
If $R\geq 1$:
\begin{equation} \label{estim6}
\| W  \indic_{\{\max(|t|,R)<|x|\}} \|_{L^2L^4}\lesssim R^{-2}.
\end{equation}

\end{lemma}
\begin{proof}
 The proof is by direct computations, using that $W$ and $\Lambda W$ are bounded and of order $1/|x|^4$ at infinity. We sketch the proof of \eqref{estim5}. By scaling, we can assume $\lambda=1$. Then
\begin{multline*}
\left\|W \indic_{\{R+|t|<|x|<R'+|t|\}}\right\|_{L^2L^4}\lesssim \left\|\indic_{\{|t|<1\}} \indic_{\{R+|t|<|x|<R'+|t|\}}\right\|_{L^2L^4}\\
+\left\|\frac{1}{|x|^4}\indic_{\{|t|>1\}} \indic_{\{R+|t|<|x|<R'+|t|\}}\right\|_{L^2L^4}\lesssim \left( \frac{R'-R}{\lambda} \right)^{1/4}.
\end{multline*}
To prove \eqref{estim6}, we decompose:
\begin{multline*}
\| W  \indic_{\{\max(|t|,R)<|x|\}} \|_{L^2L^4}^2 =\int_{|t|\leq R}\| W \indic_{\{|x|\geq R\}} \|_{L^4}^2 dt+\int_{|t|\geq R}\| W \indic_{\{|x|\geq |t|\}} \|_{L^4}^2 dt\\
\lesssim \int_{|t|\leq R}R^{-5}dt+\int_{|t|\geq R}|t|^{-5} dt \lesssim R^{-4}.
\end{multline*}
\end{proof}

\end{appendix}

\bibliographystyle{alpha} %acm
%\bibliography{/home/duyckaerts/ownCloud2/Recherche/toto} %A la maison
%\bibliography{/users/duyckaer/ownCloud/Recherche/toto} %A Paris 13
\bibliography{channels-bibliography}
\end{document}